\documentclass[11pt]{amsart}
\usepackage{amssymb}
\usepackage{amsthm}
\usepackage{amsmath}
\usepackage{latexsym}
\usepackage{comment}
\usepackage{hyperref}
\usepackage{cleveref}
\usepackage{verbatim}
\usepackage{xypic}
\usepackage{graphicx}
\usepackage{subcaption}
\usepackage[utf8]{inputenc}
\usepackage[margin=1in]{geometry}
\usepackage[shortlabels]{enumitem}
\usepackage{xcolor}
\usepackage{stmaryrd}

\newtheorem{thm}{Theorem}[section]
\newtheorem{prop}[thm]{Proposition}

\newtheorem{lemma}[thm]{Lemma}
\newtheorem{cor}[thm]{Corollary}
\newtheorem{dfn}[thm]{Definition}

\theoremstyle{definition}

\theoremstyle{definition} \newtheorem{rmk}[thm]{Remark}

\newtheorem{conj}[thm]{Conjecture}
\newtheorem{asr}[thm]{Assertion}

\newcommand{\cc}{\mathbb{C}}
\newcommand{\rr}{\mathbb{R}}
\newcommand{\qq}{\mathbb{Q}}
\newcommand{\zz}{\mathbb{Z}}

\newcommand{\proj}{\mathbb{P}}

\newcommand{\PGL}{\mathrm{PGL}}

\newcommand{\hvt}{[\mathrm{h.v.t.}]}

\newcommand{\Berk}{\mathbb{P}_{\cc_K}^{1, \mathrm{an}}}
\newcommand{\Smin}{S^{\mathrm{min}}}

\graphicspath{{pics/}}

\title[Branch points of split degenerate superelliptic curves I]{Branch points of split degenerate superelliptic curves II: on a conjecture of Gerritzen and van der Put}
\author{Jeffrey Yelton}
\address{\parbox{\linewidth}{Department of Mathematics and Computer Science, Wesleyan University \\ 265 Church Street, Middletown, CT 06459-0128}}
\email{jyelton@wesleyan.edu}

\begin{document}

\maketitle

\begin{abstract}

Let $K$ be a field with a discrete valuation, and let $p$ be a prime.  It is known that if $\Gamma \lhd \Gamma_0 < \PGL_2(K)$ is a Schottky group normally contained in a larger group which is generated by order-$p$ elements each fixing $2$ points $a_i, b_i \in \proj_K^1$, then the quotient of a certain subset of the projective line $\proj_K^1$ by the action of $\Gamma$ can be algebraized as a superelliptic curve $C : y^p = f(x) / K$.  The subset $S \subset K \cup \{\infty\}$ consisting of these pairs $a_i, b_i$ of fixed points is mapped bijectively modulo $\Gamma$ to the set $\mathcal{B}$ of branch points of the superelliptic map $x : C \to \proj_K^1$.  A conjecture of Gerritzen and van der Put, in the case that $C$ is hyperelliptic and $K$ has residue characteristic $\neq 2$, compares the cluster data of $S$ with that of $\mathcal{B}$.  We show that this conjecture requires a slight modification in order to hold and then prove a much stronger version of the modified conjecture that holds for any $p$ and any residue characteristic.

\end{abstract}

\section{Introduction} \label{sec intro}

This paper represents a continuation of a study of split degenerate \textit{superelliptic curves} $C$ which are $p$-cyclic covers of the projective line over a non-archimedean ground field $K$ of characteristic different from $p$ and of the uniformization of such a curve as a quotient of a certain subset of $K \cup \{\infty\}$ by the action of a free group $\Gamma < \PGL_2(K)$.  We denote the set of branch points of the covering map $C \to \proj_K^1$ by $\mathcal{B}$ and write $d = \#\mathcal{B}$.  One can easily compute from the Riemann-Hurwitz formula that the genus of $C$ is given by $\frac{1}{2}(p - 1)(d - 2)$.  It is also well known that a superelliptic curve which is a degree-$p$ cover of the projective line can be described by an affine model of the form 
\begin{equation} \label{eq superelliptic model}
y^p = \prod_{i = 1}^{d'} (x - z_i)^{r_i} \in K[x],
\end{equation}
 with $d' \in \{d, d - 1\}$, where $1 \leq r_i \leq p - 1$ for $1 \leq i \leq d'$ and $\mathcal{B} = \{z_1, \dots, z_{d'}\}$ (resp. $\mathcal{B} = \{z_1, \dots, z_{d'}, \infty\}$) if $d' = d$ (resp. $d' = d - 1$).  In the special case that $p = 2$, we call $C$ a \textit{hyperelliptic curve}; in this case, our formula for the genus $g$ implies that $d$ is even and that the equation in (\ref{eq superelliptic model}) can be written as $y^2 = f(x) \in K[x]$ for a squarefree polynomial $f$ of degree $2g + 1$ (resp. $2g + 2$) if the degree-$2$ cover $C \to \proj_K^1$ is not (resp. is) branched above $\infty$.
 
\subsection{Background on non-archimedean uniformization of superelliptic curves} \label{sec intro background}

Throughout this paper, we assume that $K$ is a field equipped with a discrete valuation $v : K^\times \to \zz$, and we denote by $\cc_K$ the completion of an algebraic closure of $K$.  We fix a prime $p$ and a primitive $p$th root of unity $\zeta_p \in \cc_K$ and assume throughout that we have $\zeta_p \in K$.  (This will ensure, among other things, that each automorphism of $C / K$ as a cyclic $p$-cover of $\proj_K^1$ is defined over $K$.)  We adopt the convention of using the notation $\proj_K^1$ both for the projective line with its structure as a variety over $K$ and for the set of $K$-points of $\proj_K^1$, \textit{i.e.} in a context that will appear frequently in this paper, we write $\proj_K^1$ for the set $K \cup \{\infty\}$.

Mumford showed in his groundbreaking paper \cite{mumford1972analytic} that any curve $C / K$ (not necessarily superelliptic) of genus $g \geq 1$ can be realized as a quotient of a certain subset $\Omega \subset \proj_K^1$ by the action of a free subgroup $\Gamma < \PGL_2(K)$ of $g$ generators via fractional linear transformations if and only if the the curve $C$ satisfies a property called \textit{split degenerate reduction} (see \cite[Definition 6.7]{papikian2013non} or \cite[\S IV.3]{gerritzen2006schottky}).  The free subgroup $\Gamma < \PGL_2(K)$ must act discontinuously on $\proj_K^1$ (\textit{i.e.} the set of limit points under its action must not coincide with all of $\proj_K^1$), and the subset $\Omega \subset \proj_K^1$ such that $C$ can be uniformized as the quotient $\Omega / \Gamma$ coincides with the set of non-limit points.  This main result on non-archimedean uniformization of curves is given as \cite[Theorem 4.20]{mumford1972analytic} and \cite[Theorems III.2.2, III.2.12.2, and IV.3.10]{gerritzen2006schottky}; many more details are contained in those sources.  We comment that in the special case of $g = 1$, after applying an appropriate automorphism of $\proj_K^1$ we get $\Omega = \proj_K^1 \smallsetminus \{0, \infty\} = K^\times$ and that $\Gamma$ is generated by the fractional linear transformation $z \mapsto qz$ for some element $q \in K^\times$ of positive valuation, and thus we recover the Tate uniformization $C \cong K^\times / \langle q \rangle$ established in \cite{tate1995review}.

It is shown in \cite[\S9.2]{gerritzen2006schottky} and \cite[\S1]{van1983non} (for the $p = 2$ case) and in \cite[\S2]{van1982galois} (for general $p$) that given a prime $p$ and a split degenerate curve $C / K$ of genus $(p - 1)g$ realized as such a quotient $\Omega / \Gamma$, the curve $C$ is superelliptic and a degree-$p$ cover of $\proj_K^1$ if and only if $\Gamma$ is normally contained in a larger subgroup $\Gamma_0 < \PGL_2(K)$ generated by $g + 1$ elements $s_0, \dots, s_g$ whose only relations are $s_0^p = \dots = s_g^p = 1$.\footnote{In fact, the genus of every split degenerate $p$-cyclic cover of $\proj_K^1$ is divisible by $p - 1$, so all such curves arise in this way.}  In this situation, we have $[\Gamma_0 : \Gamma] = p$ and $\Omega / \Gamma_0 \cong \proj_K^1$ so that the natural surjection $\Omega / \Gamma \twoheadrightarrow \Omega / \Gamma_0$ is just the degree-$p$ covering map $C \to \proj_K^1$.  Each order-$p$ element $s_i \in \PGL_2(K)$ fixes exactly $2$ points of $\proj_K^1$, which we denote as $a_i$ and $b_i$.

As in \cite{yelton2024branch}, we call the subgroup $\Gamma_0 < \PGL_2(K)$ discussed above a \emph{$p$-Whittaker group} (see \cite[Remark 2.12]{yelton2024branch} for an explanation for this terminology).  One can show that we have $a_i, b_i \in \Omega$ for $0 \leq i \leq g$: see \Cref{cor S in Omega} below.  Writing $S = \{a_0, b_0, \dots, a_g, b_g\} \subset \Omega$, it is easy to verify that the set-theoretic image of $S$ modulo the action of the $p$-Whittaker group $\Gamma_0$ coincides with the set of branch points $\mathcal{B} \subset \proj_K^1 \cong \Omega / \Gamma_0$.  For $0 \leq i \leq g$, let us write $\alpha_i, \beta_i \in \mathcal{B}$ for the respective images of $a_i, b_i \in \Omega$, so that $\mathcal{B} = \{\alpha_0, \beta_0, \dots, \alpha_g, \beta_g\}$.  By \cite[Proposition 3.1(a)]{van1982galois}, the superelliptic curve $C$ has an equation of the form 
\begin{equation} \label{eq split degenerate superelliptic}
y^p = \prod_{i = 0}^g (x - \alpha_i)^{m_i} (x - \beta_i)^{p - m_i},
\end{equation}
where the term $(x - \alpha_i)^{m_i}$ (resp. $(x - \beta_i)^{p - m_i}$) in the product is replaced by $1$ if we have $\alpha_i = \infty$ (resp. $\beta_i = \infty$).

\subsection{Our previous results} \label{sec intro previous results}

In the previous work \cite{yelton2024branch}, the author considered a construction proceeding in the other direction: after fixing a prime $p$, we begin with a subset $S \subset \proj_K^1$ of cardinality $2g + 2$ for some integer $g \geq 1$ and try to construct a superelliptic curve of genus $(p - 1)g$ over $K$ which is uniformized using a Schottky group $\Gamma \lhd \Gamma_0 = \langle s_0, \dots, s_g \rangle$, where the fixed points $a_i, b_i \in \proj_K^1$ of each generator $s_i$ of the associated $p$-Whittaker group are the elements of $S$.  It is by no means the case that an arbitrary $(2g + 2)$-element set $S$ can be partitioned into pairs $\{a_i, b_i\}$ which each consitute the set of fixed points of an order-$p$ automorphism $s_i \in \PGL_2(K)$ such that the group $\Gamma_0 := \langle s_0, \dots, s_j \rangle$ is a $p$-Whittaker group that can be used to uniformize a superelliptic curve over $K$ (\textit{i.e.} it may be the case that for every possible partition $S = \bigsqcup_{i = 0}^g \{a_i, b_i\}$ and every choice of order-$p$ automorphisms $s_i$ fixing $a_i, b_i$, the elements $s_i$ may satisfy some group relations other than $s_0^p = \dots = s_g^p = 1$, so that the subgroup of $\PGL_2(K)$ that they generate cannot be $p$-Whittaker).

Given a $(2g + 2)$-element subset $S$, a partition $S = \bigsqcup_{i = 0}^g \{a_i, b_i\}$, and a choice of $s_i \in \PGL_2(K)$ of order $p$ which fixes $a_i, b_i$ for $0 \leq i \leq g$, we defined (in \cite[Definition 1.1]{yelton2024branch}) the \emph{associated subgroups} $\Gamma \lhd \Gamma_0 < \PGL_2(K)$ as $\Gamma_0 = \langle s_0, \dots, s_g \rangle$ and $\Gamma = \langle s_0^{j - 1} s_i s_0^j \rangle_{0 \leq i \leq g, 1 \leq j \leq p - 1}$.  If $\Gamma$ is a Schottky group (thus leading to the construction of a superelliptic curve) and $\Gamma_0$ cannot be generated by fewer than $g + 1$ elements (which then ensures that $\Gamma$ is freely generated by the $(p - 1)g$ elements of the form $s_0^{j - 1} s_i s_0^j$), then we say (as in \cite[Definition 1.2]{yelton2024branch}) that $S$ is \emph{$p$-superelliptic}.  It turns out that if a set $S$ is $p$-superelliptic, then there is a unique partition $S = \bigsqcup_{i = 0}^g \{a_i, b_i\}$ such that the associated group $\Gamma$ is Schottky (regardless of the choice of order-$p$ automorphism $s_i \in \PGL_2(K)$ fixing $a_i, b_i$, which is unique up to power by \cite[Proposition 2.8(a)]{yelton2024branch}\footnote{While replacing a generator $s_i$ by a prime-to-$p$ power does not change $\Gamma_0$, it does change $\Gamma$, and it affects the curve $C \cong \Omega / \Gamma$ by changing the integer $m_i$ appearing in its defining equation (\ref{eq split degenerate superelliptic}): see \cite[Proposition 3.2]{van1982galois}.}).  Moreover, it is shown that the condition on $S$ of \emph{being clustered in $\frac{v(p)}{p - 1}$-separated pairs} (see \Cref{dfn clustered in pairs} below) is necessary (though not sufficient in general: see \cite[Example 2.18]{yelton2024branch}) for $S$ to be $p$-superelliptic.  The above two assertions are given as parts (b) and (a) respectively of \cite[Remark 2.15]{yelton2024branch}.

The main goal of \cite{yelton2024branch} is to find a method of determining whether a given $(2g + 2)$-element subset $S \subset \proj_K^1$ is $p$-superelliptic.  An algorithm is provided (as \cite[Algorithm 4.2]{yelton2024branch}) in which a non-empty even-cardinality input subset $S \subset \proj_K^1$ is transformed through a sequence of modifications called \emph{foldings} (which are bijections $\phi: S \to S' \subset \proj_K^1$ satisfying certain properties) which do not affect the associated group $\Gamma$, checking at each step that the resulting set $S'$ is clustered in $\frac{v(p)}{p - 1}$-separated pairs, until eventually, in the case that the input set $S$ is $p$-superelliptic, it is transformed into a set $\Smin$ satisfying a property which is called \emph{optimality} (see \cite[Definition 3.12]{yelton2024branch}) for which there are no more ``good foldings'' to be performed.  It is shown in particular (as \cite[Lemma 3.18]{yelton2024branch}) that an optimal set is $p$-superelliptic.

\subsection{Our main result} \label{sec intro main result}

Our main purpose in this paper is to compare a $p$-superelliptic set $S$ with its image $\mathcal{B}$ modulo the action of $\Gamma_0$ (the branch points of the resulting superelliptic curve), specifically in terms of the combinatorial data of the distances (under the metric induced by the discrete valuation $v$) between elements of $S$ and the corresponding elements in $\mathcal{B}$.  In other words, we are interested in comparing the \emph{cluster data} of $S$ with that of its modulo-$\Gamma_0$ image $\mathcal{B}$ (see \Cref{dfn cluster} below).  The main inspiration for this general question is a conjecture posed by Gerritzen and van der Put on page 282 of their book \cite{gerritzen2006schottky}, which may be paraphrased in the language of clusters (as \Cref{conj GvanderP clusters} below) as essentially stating that a subset $\mathfrak{s} \subset S$ is a cluster if and only its image modulo the action of $\Gamma_0$ is a cluster of $\mathcal{B}$.  This conjecture was formulated only in a context where $p = 2$ and the residue characteristic of $K$ is not $2$.

We are able to resolve this conjecture by showing that it is false in general as stated (see \Cref{rmk problem with conjecture} below) but that it is true when $S$ satisfies the condition of \emph{optimality} discussed in the previous subsection.  (We have shown as \cite[Corollary 3.23, Remark 3.27]{yelton2024branch} that in the special case studied by Kadziela in his dissertation \cite[Chapters 5, 6]{kadziela2007rigid}, the set $S$ is optimal, thus recovering Kadziela's result that the conjecture holds in this case.)  We assert this in the more general context of removing all conditions on $p$ and on the residue characteristic of $K$.  We are moreover able to directly compare the cluster data of the optimal set $S$ with that of its image $\mathcal{B}$.  When $p$ is not the residue characteristic of $K$, this comparison is simple to state; otherwise, a general formula relating the relative depth of a cluster of $S$ with that of its image cluster of $\mathcal{B}$ is too unpleasant to write down here, but such a formula becomes simple in certain cases such as when no cluster of $S$ is itself the union of $\geq 2$ even-cardinality sub-clusters.  Our main result (in the cases that are not too cumbersome to write down) is summed up in the following theorem.

\begin{thm}[cluster version] \label{thm main clusters}

Let $S \subset \proj_K^1$ be an optimal subset with associated $p$-Whittaker group $\Gamma_0$ and superelliptic curve $C \cong \Omega / \Gamma_0$ with branch points $\mathcal{B} \subset \proj_K^1$; and write $\pi : S \to \mathcal{B}$ for the bijection corresponding to reduction modulo $\Gamma_0$.  Assume that $\pi(b_g) = b_g = \infty$.

\begin{enumerate}[(a)]

\item A subset $\mathfrak{s} \subset S$ with $2 \leq \#\mathfrak{s} \leq 2g$ is a cluster of $S$ if and only if its image $\pi(\mathfrak{s}) \subset \mathcal{B}$ is a cluster of $\mathcal{B}$.

\item Given a cluster $\mathfrak{s}$ of $S$ with $2 \leq \#\mathfrak{s} \leq 2g$, in certain cases the relative depth of $\pi(\mathfrak{s})$ can be compared to that of $\mathfrak{s}$ as follows; below we write $\mathfrak{s}'$ for the smallest cluster of $S$ of cardinality $\geq 2$ and $\leq 2g$ which properly contains a cluster $\mathfrak{s}$.

\begin{enumerate}[(i)]
\item If $\mathfrak{s}$ has odd cardinality, then we have 
\begin{equation} \delta(\pi(\mathfrak{s})) = p\delta(\mathfrak{s}). \end{equation}
\item If $\mathfrak{s}$ has even cardinality and if $p$ is not the residue characteristic of $K$, then we have 
\begin{equation} \delta(\pi(\mathfrak{s})) = \delta(\mathfrak{s}). \end{equation}
\item If $\mathfrak{s}$ has even cardinality and if neither $\mathfrak{s}$ nor $\mathfrak{s}'$ is the union of $\geq 2$ even-cardinality sub-clusters, then we have 
\begin{equation} \delta(\pi(\mathfrak{s})) = \delta(\mathfrak{s}) + 2v(p).\end{equation}
\end{enumerate}

\end{enumerate}

\end{thm}

\begin{rmk} \label{rmk no encumbrance}

The hypothesis in the above theorem that $\beta_g = b_g = \infty$ is no encumbrance to finding the cluster data of the set $\mathcal{B}$ branch points of a superelliptic curve $C$ which is uniformized using the Schottky group associated to a set $S$ of fixed points.  If, in our situation, we have $b_g \neq \infty$, then it is easy to see that we may choose a fractional linear transformation $\sigma \in \PGL_2(K)$ such that $\sigma(b_g) = \infty$ and replace $S$ and its associated groups $\Gamma \lhd \Gamma_0$ with $\sigma(S)$ and $\sigma \Gamma \sigma^{-1} \lhd \sigma \Gamma_0 \sigma^{-1}$ respectively and that the conjugate $\sigma \Gamma_0 \sigma^{-1}$ can be used to uniformize the same curve $C$.  Meanwhile, if we have $\beta_g \neq \infty$, then we may similarly choose a fractional linear transformation $\tau \in \PGL_2(K)$ such that $\tau(\beta_g) = \infty$ and replace $\mathcal{B}$ with $\sigma(\mathcal{B})$; applying the automorphism $\tau$ to the projective line $\proj_K^1$ (in order to modify the set of branch points in this way) induces an isomorphism between models of $C$ given by the equation in (\ref{eq split degenerate superelliptic}) in terms of sets $\mathcal{B}$ of branch points.  Meanwhile, there is an easy formula relating the cluster data of a finite subset $A \subset \proj_K^1$ with that of the subset $\sigma(A) \subset \proj_K^1$ for any automorphism $\sigma \in \PGL_2(K)$ (see, for instance, \cite[Remark 5.7]{fiore2023clusters}).

\end{rmk}

We now present our main result not in terms of clusters but in the language of convex hulls in the Berkovich projective line, which removes any need to restrict to certain hypotheses on a given even-cardinality cluster $\mathfrak{s}$ or to assume that $\infty$ lies in either $S$ or $\mathcal{B}$.

\begin{thm}[Berkovich version] \label{thm main berk}

Let $S \subset \proj_K^1$ be an optimal subset with associated $p$-Whittaker group $\Gamma_0$ and superelliptic curve $C \cong \Omega / \Gamma_0$ with branch points $\mathcal{B} \subset \proj_K^1$; and write $\pi : S \to \mathcal{B}$ for the bijection corresponding to reduction modulo $\Gamma_0$.  Viewing $S \subset \Berk$ as a subset consisting of points of Type I in the Berkovich projective line $\Berk$, write $\Sigma_S$ for the convex hull of $S$ (\textit{i.e.} the smallest connected subspace of $\Berk$ which contains $S$), and define the convex hull $\Sigma_{\mathcal{B}}$ analogously.

The map $\pi$ extends to a homeomorphism $\pi_*: \Sigma_S \to \Sigma_{\mathcal{B}}$ which affects distances between points (with respect to the hyperbolic metric -- see \Cref{dfn berk} below) according to the following formula.  For $0 \leq i \leq g$, let $\Lambda_{(i)} \subset \Berk$ denote the (unique) non-backtracking path connecting the points $a_i$ and $b_i$.  For any points $v, w \in \Berk$ of Type II or III, let $[v, w] \subset \Berk$ denote the (unique) non-backtracking path connecting them, and let $\llbracket v, w \rrbracket \subseteq [v, w]$ be the subspace consisting of points $\eta$ of distance $\leq \frac{v(p)}{p - 1}$ from one of the paths $\Lambda_{(i)}$.  The subspace $\llbracket v, w \rrbracket$ is a disjoint union of shorter paths in $\Berk$; define $\mu(v, w)$ to be the sum of the lengths (i.e. distances between endpoints) of these segments.  For any points $v, w \in \Sigma_S \smallsetminus \{\eta_{a_i}, \eta_{b_i}\}_{0 \leq i \leq g}$, we have 
\begin{equation} \label{eq dilation}
\delta(\pi_*(v), \pi_*(w)) = \delta(v, w) + (p - 1) \mu(v, w).
\end{equation}

\end{thm}

\begin{rmk} \label{rmk dilation}

It follows from\cite[Remark 2.15]{yelton2024branch} or from \Cref{prop clustered in pairs} below that the set $S$ in the statement of \Cref{thm main berk}, by being $p$-superelliptic, satisfies that the axes $\Lambda_{(i)}$ are pairwise disjoint and at a distance of $> \frac{2v(p)}{p - 1}$.  The formula in \Cref{thm main berk} relating distances between points of $\Sigma_S$ to distances between their images in $\Sigma_{\mathcal{B}}$ can be described more visually as follows.  The map $\pi_*$ ``transforms'' the space $\Sigma_S$ into the space $\Sigma_{\mathcal{B}}$ simply by \textit{dilating} each axis $\Lambda_{(i)}$ by a factor of $p$, if $p$ is not the residue characteristic, and leaving the rest of the space unchanged with respect to the metric.  This generalizes to the case of residue characteristic $p$ by, instead of dilating only each axis $\Lambda_{(i)}$, dilating the \textit{tubular neighborhood} of radius $\frac{v(p)}{p - 1}$ of each axis $\Lambda_{(i)}$ by a factor of $p$.

\end{rmk}

The fact that \Cref{thm main berk} implies \Cref{thm main clusters} will be given as \Cref{prop berkovich version implies cluster version} below, while the results in \S\ref{sec berk convex hulls} will show that \Cref{thm main berk} allows us to compute the relative depths of clusters of $\mathcal{B}$ in the cases not covered by the statement of \Cref{thm main clusters}, so that \Cref{thm main berk} is stronger than \Cref{thm main clusters}.  We also mention that, as discussed in \S\ref{sec intro previous results}, one may use \cite[Algorithm 4.2]{yelton2024branch} to turn any $p$-superelliptic set into an optimal set; therefore, \Cref{thm main berk} enables us, given \textit{any} $p$-superelliptic set $S$, to determine the metric graph isomorphism type of $\Sigma_{\mathcal{B}}$ (or, equivalently, the cluster data of $\mathcal{B}$ by the results of \S\ref{sec berk convex hulls},\ref{sec berk dictionary} below), where $\mathcal{B} \subset \proj_K^1$ is the set of branch points of the superelliptic curve associated to $S$.

\subsection{Outline of the paper} \label{sec intro outline}

We postpone the proof of our main result to \S\ref{sec proof of main} and use \S\ref{sec GvanderP},\ref{sec berk} to establish three different ways to discuss the non-archimedean combinatorial properties of a finite subset $A \subset \proj_K^1$ and how to translate between them, in order to understand the relationship between the conjecture as originally stated by Gerritzen and van der Put and our two variants of it in the form of Theorems \ref{thm main clusters} and \ref{thm main berk}.  These three aspects of the set $A$ are its \textit{position}, its \textit{cluster data}, and the \textit{metric graph properties of its convex hull} in the Berkovich projective line.  Each of these frameworks has its advantages and its disadvantages.  Looking at the position of $A$ is the framework used by Gerritzen and van der Put in many places in their book \cite{gerritzen2006schottky} and is directly tied to constructing useful models of the projective line, but position conveys less information than cluster data and the convex hull do (specifically, the depths of clusters and the lengths of segments of the convex hull are not reflected by the position).  The language of cluster data has recently become popular as it relates to many applications involving the arithmetic of hyperelliptic and superelliptic curves (\cite{dokchitser2019semistable, dokchitser2022arithmetic} being the originating examples), and cluster data is easy to immediately compute, but it has the downside of being somewhat affected by automorphisms of the projective line (albeit in a way that can be described by easy formulas) while position and metric graph properties of the convex hull are independent of the chosen coordinate of the projective line; it is for this reason that our results stated in terms of clusters often require an extra hypothesis to be articulated succinctly.  Moreover, some hypotheses and results are cumbersome to describe in terms of clusters, which is why only certain cases are fully described in the statement of \Cref{thm main clusters}.  Viewing a set $A$ through its convex hull and studying its metric graph properties, while less accessible of a framework than that of position or cluster data, is the most powerful both in terms of ease of stating hypotheses and results and ease of proving them directly.  It is for this reason that the result that we will directly prove is \Cref{thm main berk} and our methods of proving it mainly involve convex hulls and other subspaces of the Berkovich projective line.

In \S\ref{sec GvanderP}, we examine the aforementioned conjecture of Gerritzen and van der Put, putting it as precisely as possible by rigorously defining what they mean by ``position'' of a finite subset of $\proj_K^1$ (in \S\ref{sec GvanderP position}) and then translating their conjecture into the language of clusters (in \S\ref{sec GvanderP clusters}).  We then show in \S\ref{sec GvanderP counterexample} that without adding the hypothesis that the set $S$ is optimal as in \Cref{thm main clusters}, the conjecture does not hold.

The object of \S\ref{sec berk} is then to reframe everything in terms of the convex hull $\Sigma_S$ of the subset $S \subset \proj_K^1$ as a subset of the Berkovich projective line $\Berk$ consisting of points of Type I, culminating in a proof that the ``Berkovich version'' of our main result (\Cref{thm main berk}) implies the ``cluster version'' of our main result (\Cref{thm main clusters}).  This is done by introducing (a slightly simplified version of) the Berkovich projective line $\Berk$ in \S\ref{sec berk berk}, studying convex hulls in $\Berk$ in \S\ref{sec berk convex hulls}, and directly relating the metric graph structure of the convex hull of $A$ to its cluster data in \S\ref{sec berk dictionary}.

The actual proof of \Cref{thm main berk} takes up \S\ref{sec proof of main}, which is the heart of this paper and broken up into five subsections, the first two of which provide background information about the action of $\PGL_2(K)$ on convex hulls and the map $\pi : \Omega \twoheadrightarrow \Omega / \Gamma \cong C$ as an explicit theta function, and the last three of which break the proof of \Cref{thm main berk} into three parts.  Along the way, we will prove a result (\Cref{thm one segment}) which approximates the outputs of $\pi$ on a certain subset of its domain $\Omega$, which is interesting in its own right.

We finish the paper with a corollary to our main result which describes a property that generally holds for the branch points of a split degenerate $p$-cyclic cover of the projective line.

\subsection{Acknowledgements}

The author is grateful to Christopher Rasmussen for helpful discussions that took place during the process of developing these results.  The author would also like to thank Robert Benedetto for conversations which enabled him to frame the main result in terms of map on subsets of the Berkovich projective line which is naturally induced by the uniformizing map of the superelliptic curve in the proof of \Cref{thm main berk} and to make the proof more rigorous.

\section{A conjecture of Gerritzen and van der Put} \label{sec GvanderP}

On \cite[p. 282]{gerritzen2006schottky}, Gerritzen and van der Put make a conjecture regarding the relationship between the set $S$ of fixed points of generators of a $p$-Whittaker group $\Gamma_0$ and its image $\mathcal{B}$ under reduction modulo $\Gamma_0$, which is the set of branch points of the resulting superelliptic curve.  The conjecture is worded so as to say that the ``position'' of $S$ and the ``position'' of $\mathcal{B}$ are ``identical''.  This conjecture is restated by Kadziela in his dissertation as \cite[Conjecture 3.1]{kadziela2007rigid}.  In \S\ref{sec GvanderP position}, we make the statement of this conjecture rigorous by carefully defining \emph{position}, and in \S\ref{sec GvanderP clusters}, we introduce the language of clusters and translate the conjecture so that it can be expressed in this language.  Then in \S\ref{sec GvanderP counterexample}, we show by counterexample that the conjecture is not actually true.  As it will turn out that the conjecture becomes true with the addition of a hypothesis on $S$, our work in this section is still important in allowing us to properly define what it means to ``have identical position'' and to express it in the language of clusters.

\subsection{Statement of the conjecture in terms of position} \label{sec GvanderP position}

Although the term ``position'' is not defined in either \cite{gerritzen2006schottky} or \cite{kadziela2007rigid} in very precise language, one can understand the meaning to be as follows.  Let $R \subset K$ be the ring of integers, and let $k$ be the residue field of the local ring $R$.  Given an ordered $3$-element subset $\underline{z} = \{z_0, z_1, z_\infty\} \subset \proj_K^1$, there is a unique automorphism $\gamma_{\underline{z}} \in \PGL_2(K)$ which sends $z_i$ to $i \in \proj_K^1$ for $i = 0, 1, \infty$.  Composing this with the reduction map $R \to k$, we get a map $\bar{\gamma}_{\underline{z}} : \proj_K^1 \to \proj_k^1$.  In \cite[\S I.2]{gerritzen2006schottky}, Gerritzen and van der Put define a tree $T(A)$ whose vertices correspond to equivalence classes of ordered $3$-element subsets $\underline{z}$ of $A$, where subsets $\underline{z}, \underline{w} \subset \proj_K^1$ are equivalent if the automorphism $\bar{\gamma}_{\underline{w}} \bar{\gamma}_{\underline{z}}^{-1} \in \PGL_2(k)$ is invertible.  Writing $|T(A)|$ for the set of vertices of the graph $T(A)$, let us define the map 
\begin{equation*}
R_A : \proj_K^1 \to (\proj_k^1)^{\#|T(A)|}
\end{equation*}
to be the product of the maps $\bar{\gamma}_{\underline{z}}$, where $\underline{z}$ ranges over a set of representatives of each equivalence class of ordered $3$-element subsets $\underline{z}$ of $A$.  The image $R_A(\proj_K^1)$ is a curve over $k$ whose components are all isomorphic to $\proj_k^1$ and intersect only at ordinary double points, none of which is the image of a point in $A$; the components correspond to the vertices of the graph $T(A)$.  See \cite[\S I.4.2]{gerritzen2006schottky} for more details and proofs of these assertions.  (We moreover note that the image of $R_A$ is in fact the special fiber of a model of $\proj_R^1$ which is minimal with respect to the condition that the images of the $R$-points of this model extending the points in $A$ do not intersect in the special fiber.)  We may now define ``position'' as follows.

\begin{dfn} \label{dfn position}

The \emph{position} of a finite subset $A \subset \proj_K^1$ of cardinality $\geq 3$ is the combinatorial data of the tree $T(A)$ along with the map $r_A: A \to |T(A)|$ given by sending a point $z \in A$ to the vertex of $T(A)$ corresponding to the (unique) component of $R_A(\proj_K^1)$ which contains $R_A(z)$.

Given two finite subsets $A, A' \subset \proj_K^1$ of (equal) cardinality $\geq 3$ and a bijection $\varphi : A \to A'$, the sets $A$ and $A'$ are said to \emph{have the same position} if there is an graph isomorphism $T(A) \stackrel{\sim}{\to} T(A')$ making the below diagram commute.
\vspace{-1em}
\begin{equation} \label{eq position}
\xymatrix{ A \ar[r]^{\!\!\!\!r_A} \ar[d]^\varphi & |T(A)| \ar[d]^\wr 
\\ A' \ar[r]^{\!\!\!r_{A'}} & |T(A')| }
\end{equation}

\end{dfn}

It is clear directly from definitions that the position of a subset $A \subset \proj_K^1$ does not change after applying an automorphism in $\PGL_2(K)$ to the whole subset.  Now the conjecture of Gerritzen and van der Put may be presented as follows.

\begin{conj}[Gerritzen and van der Put, 1980] \label{conj GvanderP}

With the above set-up, the $(2g + 2)$-element subsets $S \subset \proj_K^1$ and $\mathcal{B} \subset \proj_K^1$ have the same position.

\end{conj}

\subsection{Cluster data and position} \label{sec GvanderP clusters}

We introduce the language of \emph{clusters} and \emph{cluster data}, following its use in \cite{dokchitser2022arithmetic}, below.

\begin{dfn} \label{dfn cluster}

Let $A \subset \proj_K^1$ be a finite subset.  A subset $\mathfrak{s} \subseteq A$ is called a \emph{cluster} (of $A$) if there is some subset $D \subset K$ which is a disc under the metric induced by the valuation $v : K^\times \to \zz$ such that $\mathfrak{s} = A \cap D$.  The \emph{depth} of a cluster $\mathfrak{s}$ is the integer 
\begin{equation}
d(\mathfrak{s}) := \min_{z, z' \in \mathfrak{s}} v(z - z').
\end{equation}
Given a cluster $\mathfrak{s}$ of $A$ which is properly contained in another cluster of $A$, and letting $\mathfrak{s}' \subseteq A$ be the minimum cluster containing $\mathfrak{s}$, we define the \emph{relative depth} $\delta(\mathfrak{s})$ of $\mathfrak{s}$ to be the difference $d(\mathfrak{s}) - d(\mathfrak{s}')$.

The data of all clusters of a finite subset $A \subset \proj_K^1$ along with each of their depths (or relative depths) is called the \emph{cluster data} of $A$.  The data that only consists of all clusters of $A$ (that is, which subsets of $A$ are clusters, without considering depth) is called the \emph{combinatorial cluster data} of $A$.

\end{dfn}

\begin{rmk} \label{rmk reciprocal}

Applying a fractional linear transformation to a finite subset $A \subset \proj_K^1$ changes the combinatorial cluster data in a predictable way.  Every fractional linear transformation is a composition of translations, homotheties, and the reciprocal map $\iota: z \mapsto z^{-1}$.  Transformations and homotheties clearly do not affect the combinatorial cluster data, while the reciprocal map affects it if and only if not all elements of $A$ have the same valuation, in the following way.  Let $\mathfrak{s} \subset A$ be the subset of elements with maximal valuation; it is easy to see that $\mathfrak{s}$ is a cluster.  Then for each cluster $\mathfrak{c}$ of $A$ such that $\mathfrak{c} \cap \mathfrak{s} = \varnothing$ or $\mathfrak{c} \subsetneq \mathfrak{s}$, the image $\iota(\mathfrak{c})$ is a cluster of $\iota(A)$, while for each cluster $\mathfrak{c}$ of $A$ which contains $\mathfrak{s}$, the image $\iota(A \smallsetminus \mathfrak{c})$ is a cluster of $\iota(A)$, and this describes all clusters of $\iota(A)$.  It is also easy to describe what happens to the relative depths of the clusters; see \cite[Remark 5.7]{fiore2023clusters} for more details.

\end{rmk}

\begin{lemma} \label{lemma position and cluster data}

In the definition of $T(A)$ from \S\ref{sec GvanderP position}, each equivalence class of ordered triples has a representative $(z_0, z_1, z_\infty)$ satisfying $v(z_\infty - z_0) < v(z_1 - z_0)$, where we adopt the convention that $v(\infty - z) = \infty$ for $z \in K$.  Conversely, any ordered triple $(z_0', z_1', z_\infty')$ such that $v(z_\infty' - z_0') < v(z_1' - z_0') = v(z_1 - z_0) = v(z_i' - z_j)$ for all $i, j \in \{0, 1\}$ is equivalent to $(z_0, z_1, z_\infty)$.

\end{lemma}

\begin{proof}

Choose any ordered triple $\underline{w} = (z_0, z_1, w)$.  It is easy to verify, first of all, that any permutation of the coordinates $z_0, z_1, w$ of the ordered triple defining $\underline{z}$ does not affect the equivalence class, so after applying a suitable permutation, we assume that $v(z_1 - z_0) \geq \max\{v(w - z_0), v(w - z_1)\}$.  Choose an element $z_\infty \in A$ satisfying $v(z_\infty - z_0) = v(z_\infty - z_1) < v(z_1 - z_0)$ (such a $z_\infty$ certainly exists since we may take $z_\infty = \infty$), and let $\underline{z}$ be the ordered triple $(z_0, z_1, z_\infty)$.  It is straightforward to verify that the map $\bar{\gamma}_{\underline{z}} : \proj_K^1 \to \proj_k^1$ may be described as sending $z \in \proj_K^1$ to the reduction of $(z_1 - z_0)^{-1}(z - z_0)$.  Now it is visible from this formula that we have $\bar{\gamma}_{\underline{z}}(w) \notin \{1, 0\}$, so $\bar{\gamma}_{\underline{z}}$ sends $z_0, z_1, w$ to $3$ distinct points in $\proj_k^1$.  It follows that the composition $\bar{\gamma}_{\underline{z}} \bar{\gamma}_{\underline{w}}^{-1} \in \mathrm{M}_2(k)$ sends $0, 1, \infty \in \proj_k^1$ to $3$ distinct points.  It is now an elementary exercise to show that $\bar{\gamma}_{\underline{z}} \bar{\gamma}_{\underline{w}}^{-1}$ is invertible, so that $\underline{z}$ and $\underline{w}$ are in the same equivalence class.

Now let $z_\infty', z_1' \in A$ be elements satisfying $v(z_\infty' - v_0) < v(z_1' - z_0) = v(z_1 - z_0) = v(z_1' - z_1)$.  Then we apply what we have just shown to conclude that the ordered triple $(z_0, z_1, z_1')$ is equivalent both to $(z_0, z_1, z_\infty)$ and to $(z_0, z_1', z_\infty')$, and therefore the last two ordered triples are equivalent to each other.  Then by a similar argument, for any $z_0' \in A$ satisfying the hypotheses of the lemma, we see that the ordered triple $(z_0', z_1', z_\infty')$ is equivalent to all of these.
\end{proof}

\begin{prop} \label{prop position and cluster data}

Let $\varphi: A \to A'$ be a bijection of finite subsets of $\proj_K^1$, and assume that we have $\infty \in A \cap A'$ and that $\varphi(\infty) = \infty$.  Then the sets $A$ and $A'$ have the same position if and only if $\varphi$ acts as a bijection between the clusters of $A$ and those of $A'$.

\end{prop}

\begin{proof}

There is a correspondence between the vertices of $|T(A)|$ and the clusters of $A$ of cardinality $\geq 2$, described as follows.  By \Cref{lemma position and cluster data}, a vertex $v \in |T(A)|$ is represented by an ordered triple $(z_0, z_1, z_\infty)$ satisfying $v(z_\infty - z_0) < v(z_1 - z_0)$; the corresponding cluster is the smallest cluster $\mathfrak{s}$ which contains $z_0$ and $z_1$ (note that $z_\infty \notin \mathfrak{s}$).  Conversely, given a cluster $\mathfrak{s}$ of $A$, choosing elements $z_0, z_1 \in \mathfrak{s}$ such that $z_0$ and $z_1$ do not both lie in a proper sub-cluster of $\mathfrak{s}$, and choosing $z_\infty \in A \smallsetminus \mathfrak{s}$, the corresponding vertex is the one represented by the ordered triple $(z_0, z_1, z_\infty)$.  \Cref{lemma position and cluster data} directly implies that this vertex depends neither on the choice of elements $z_0, z_1 \in \mathfrak{s}$ which do not both lie in a proper sub-cluster of $\mathfrak{s}$ nor on the choice of $z_\infty \in A \smallsetminus \mathfrak{s}$.

Let us denote the cluster corresponding to a vertex $v \in |T(A)|$ by $\mathfrak{s}_v$.  Given any ordered triple $\underline{z} = (z_0, z_1, z_\infty)$ satisfying $v(z_\infty - z_0) < v(z_1 - z_0)$ which represents a vertex $v \in |T(A)|$ and given any pair of distinct elements $z, z' \in A$, the map $\bar{\gamma}_{\underline{z}}$ (whose formula is explicitly given in the proof of \Cref{lemma position and cluster data}) sends $z$ and $z'$ to the same point in $\proj_k^1$ if and only if either we have $z, z' \notin \mathfrak{s}_v$ or we have $z, z' \in \mathfrak{c}$ for some proper sub-cluster $\mathfrak{c} \subsetneq \mathfrak{s}_v$.  It follows that for each vertex $v \in |T(A)|$, the inverse image $r_A^{-1}(v)$ coincides with the set of elements $z \in \mathfrak{s}_v$ such that $z$ is not in any proper sub-cluster of $\mathfrak{s}_v$, together with the element $\infty \in A$ if $\mathfrak{s}_v$ is not properly contained in any cluster.  Therefore, the map $r_A$ determines and is determined by the combinatorial cluster data of $A$, and the claim of the proposition follows.
\end{proof}

In light of \Cref{prop position and cluster data} above, and keeping in mind that a given element in a finite subset $A \subset \proj_K^1$ may be moved to $\infty$ by applying an appropriate fractional linear transformation to $A$, \Cref{conj GvanderP} may be restated in the language of cluster data as follows.

\begin{asr} \label{conj GvanderP clusters}

With the above set-up, after possibly applying a suitable fractional linear transformation to the $(2g + 2)$-element subset $S \subset \proj_K^1$ and a suitable fractional linear transformation to its image $\pi(S) = \mathcal{B} \subset \proj_K^1$, the subsets $S, \mathcal{B} \subset \proj_K^1$ have the same combinatorial cluster data.

\end{asr}

\subsection{Counterexample to the conjecture as stated} \label{sec GvanderP counterexample}

Gerritzen and van der Put show in \cite[\S IX.2.5]{gerritzen2006schottky} that their conjecture holds as stated when $g = 1$, and they claim, without showing explicit calculations, that they have confirmed by checking each possible position of a $6$-element subset $S \subset \proj_K^1$ that their conjecture holds also when $g = 2$.  However, in the following remark we show that there are counterexamples \Cref{conj GvanderP clusters} even when $g = 2$.

\begin{rmk} \label{rmk problem with conjecture}

We observe an issue with Gerritzen and van der Put's conjecture (as interpreted literally) in that when $g \geq 2$ it is possible for two good cardinality-$(2g + 2)$ subsets $S, S' \in \proj_K^1$ to induce the same Schottky group $\Gamma \lhd \Gamma_0$ but to not have the same position.  As an example, suppose that $K = \qq_3(\zeta_3)$; let 
\begin{equation*}
S = \{a_0 := -9, b_0 := 9, a_1 := 3, b_1 := 12, a_2 := 1, b_2 := \infty\};
\end{equation*}
 and for $i = 0, 1, 2$, let $s_i \in \PGL_2(K)$ be the unique fractional linear transformation of order $2$ which fixes the points $a_i, b_i \in \proj_K^1$, noting that $s_0$ is simply the function $z \mapsto \frac{81}{z}$.  This set $S$ satisfies the hypotheses of \cite[Theorem 5.7]{kadziela2007rigid}, and so that theorem tells us that $\Gamma := \langle s_0 s_1, s_0 s_2 \rangle$ is a Schottky group (this can also be deduced using \cite[Corollary 3.23]{yelton2024branch}).  Now let 
\begin{equation*}
S' = \{a_0' := a_0 = -9, b_0' := b_0 = 9, a_1' := s_0(a_1) = 27, b_1' := s_0(b_1) = 27/4, a_2' := a_2 = 1, b_2' := b_2 = \infty\}.
\end{equation*}
Let $\phi: S \to S'$ be the bijection given by $(a_i, b_i) \mapsto (a_i', b_i')$ for $1 \leq i \leq 3$, noting that we have $\pi \circ \phi = \pi$ on $S$.  Letting $s_i' = s_i$ for $i = 0, 2$ and $s_1' = s_0 s_1 s_0$, we see that each $s_i'$ is the unique fractional linear transformation of order $2$ which fixes the points $a_i', b_i' \in \proj_K^1$ and that, in constructing a Schottky group from $S'$ in the usual way, we get 
\begin{equation}
\Gamma' := \langle s_0' s_1', s_0' s_2' \rangle = \langle (s_0 s_1)^{-1}, s_0 s_2 \rangle = \Gamma.
\end{equation}

However, \Cref{conj GvanderP} cannot hold for $S$ because there is no fractional linear transformation $\sigma \in \PGL_2(K)$ such that $\sigma(S')$ has the same combinatorial cluster data as $S$ (or more precisely, such that the bijection $\sigma \circ \phi : S \to \sigma(S')$ acts as a bijection between clusters of $S$ and clusters of $S'$).  To see this, we note that the even-cardinality clusters of $S$ are $\{-9, 9, 3, 12\}$, $\{-9, 9\}$, and $\{3, 12\}$, while the even-cardinality clusters of $S'$ are only $\{-9, 9, 27, 27/4\}$ and $\{27, 27/4\}$.  \Cref{rmk reciprocal} then shows that there is no fractional linear transformation $\sigma \in \PGL_2(K)$ such that $\sigma(S')$ has the same combinatorial cluster data as $S$ (or more precisely, such that the composition $\sigma \circ \phi : S \to \sigma(S')$ acts as a bijection between clusters of $S$ and clusters of $\sigma(S')$).  Alternately, in terms of position, one sees that the positions of $S$ and $S'$ fall under the cases (a) and (b) respectively in \cite[\S IX.2.5.3]{gerritzen2006schottky}.

\end{rmk}

\section{Convex hulls in the Berkovich projective line} \label{sec berk}

Given the completion of an algebraic closure $\cc_K$ of $K$, we write $v : \cc_K \to \rr$ for an extension of the valuation $v : K^\times \to \zz$.  Below when we speak of a \emph{disc} $D \subset \cc_K$, we mean that $D$ is a closed disc with respect to the metric induced by $v : \cc_K \to \rr$; in other words, $D = \{z \in \cc_K \ | \ v(z - c) \geq r\}$ for some center $c \in \cc_K$ and real number $r \in \rr$, which is the \emph{(logarithmic) radius} of $D$.  Given a disc $D \subset \cc_K$, we denote its logarithmic radius by $d(D)$.

\subsection{The Berkovich projective line and related notation} \label{sec berk berk}

The \emph{Berkovich projective line} $\Berk$ over an algebraic closure of a discrete valuation field $K$ is a type of rigid analytification of the projective line $\proj_{\cc_K}^1$ and is typically defined in terms of multiplicative seminorms on $\cc_K[x]$ as in \cite[\S1]{baker2008introduction} and \cite[\S6.1]{benedetto2019dynamics}.  Points of $\Berk$ are identified with multiplicative seminorms which are each classified as Type I, II, III, or IV.  For the purposes of this paper, as in \cite{yelton2024branch}, we may safely ignore points of Type IV and need only adopt a fairly rudimentary construction which does not directly involve seminorms.

\begin{dfn} \label{dfn berk}

Define the \textit{Berkovich projective line}, denoted $\Berk$, to be the topological space with points and topology given as follows.  The points of $\Berk$ are identified with  
\begin{enumerate}[(i)]
\item $\cc_K$-points $z \in \proj_{\cc_K}^1$, which we will call \emph{points of Type I}; and 
\item discs $D \subset \cc_K$; if $d(D) \in \qq$ (resp. $d(D) \notin \qq$), we call this a \emph{point of Type II} (resp. a \emph{point of Type III}).
\end{enumerate}

A point of $\Berk$ which is identified with a point $z \in \proj_{\cc_K}^1$ (resp. a disc $D \subset \cc_K$) is denoted $\eta_z \in \Berk$ (resp. $\eta_D \in \Berk$).

We define an infinite metric on $\Berk$ given by the distance function 
\begin{equation*}
\delta : \Berk \times \Berk \to \rr \cup \{\infty\}
\end{equation*}
defined as follows.  We set $\delta(\eta_z, \eta') = \infty$ for any point $\eta_z$ of Type I and any point $\eta' \neq \eta_z \in \Berk$.  Given a containment $D \subseteq D' \subset \cc_K$ of discs, we set $\delta(\eta_D, \eta_{D'}) = d(D) - d(D') \in \rr$.  More generally, if $D, D' \subset \cc_K$ are discs and $D'' \subset \cc_K$ is the smallest disc containing both $D$ and $D'$, we set 
\begin{equation}
\delta(\eta_D, \eta_{D'}) = \delta(\eta_D, \eta_{D''}) + \delta(\eta_{D'}, \eta_{D''}) = d(D) + d(D') - 2d(D'').
\end{equation}

We endow the subspace of $\Berk$ consisting of points of Type II and III with the topology induced by the metric given by $\delta$, and we extend this to a topology on all of $\Berk$ in such a way that, given any element $z \in \cc_K$ and disc $D \subset \cc_K$ containing $z$, the map $\lambda: [0, e^{d(D)}] \to \Berk$ given by sending $0$ to $\eta_z$ and sending $s \in (0, e^{d(D)}]$ to the disc of radius $-\ln(s)$ containing $z$ provides a path from $\eta_z$ to $\eta_D$ and that there is a similarly defined path from $\eta_D$ to $\eta_\infty$ -- see \cite[Definition 2.1, Remark 2.2]{yelton2024branch} for details.

\end{dfn}

As is discussed in \cite[Remark 2.2]{yelton2024branch}, the space $\Berk$ is path-connected, and there is a unique non-backtracking path between any pair of points in $\Berk$.  This allow us to set the following notation.  Below we denote the image in $\Berk$ of the non-backtracking path between two points $\eta, \eta' \in \Berk$ by $[\eta, \eta'] \subset \Berk$, and we will often refer to this image itself as ``the path'' from $\eta$ to $\eta'$; note that with this notation we have $[\eta, \eta'] = [\eta', \eta]$.  The above observations imply that, given a point $\eta \in \Berk$ and a subspace $\Lambda \in \Berk$, there is a (unique) point $\xi \in \Lambda$ such that every path from $\eta$ to a point in $\Lambda$ contains $\xi$; we will often speak of ``the closest point in $\Lambda$ to $\eta$'' in referring to this point $\xi$ (note that we will use this language even in the case that the distance from $\eta$ to any point in $\Lambda$ is infinite, as in when $\eta \notin \Lambda$ is of Type I).  In a similar way, if $\Lambda, \Lambda' \in \Berk$ are subspaces, we will speak of ``the closest point in $\Lambda$ to $\Lambda'$'' (and vice versa).  Given a point $\eta \in \Berk$ and subspaces $\Lambda, \Lambda' \in \Berk$, we write $\delta(\eta, \Lambda)$ (resp. $\delta(\Lambda, \Lambda')$) for the distance between $\eta$ and the closest point in $\Lambda'$ to $\eta$ (resp. between the closest point in $\Lambda$ to $\Lambda'$ and the closest point in $\Lambda'$ to $\Lambda$).

Given points $\eta, \eta' \in \Berk$, let $\eta \vee \eta' \in [\eta, \eta'] \subset \Berk$ be the point of Type II or III corresponding to the largest disc among discs corresponding to points in the path $[\eta, \eta']$.

We refer to the path between distinct points $\eta_a, \eta_b \in \Berk$ of Type I as an \emph{axis} and denote the axis connecting them by $\Lambda_{a, b} := [\eta_a, \eta_b] \subset \Berk$.  We note that the space $\Lambda_{a, b} \smallsetminus \{\eta_a, \eta_b\}$ consists precisely of those points whose corresponding disc $D$ either satisfies $\#(D \cap \{a, b\}) = 1$ or is the smallest disc containing $\{a, b\}$, a fact that we will freely use in arguments below.  In our frequent context of dealing with a $(2g + 2)$-set consisting of elements labeled $a_0, b_0, \dots, a_g, b_g$, we write $\Lambda_{(i)} = \Lambda_{a_i, b_i}$ for $0 \leq i \leq g$.

As in \cite{yelton2024branch}, given a subspace $\Lambda \subset \Berk$ and a real number $r > 0$, we define the \emph{(closed) tubular neighborhood} of $\Lambda$ of radius $r$ to be 
\begin{equation} \label{eq tubular neighborhood}
B(\Lambda, r) = \{\eta \in \Berk \ | \ \delta(\eta, \Lambda) \leq r\},
\end{equation}
 and we write $\hat{\Lambda}_{(i)} = B(\Lambda_{(i)}, \frac{v(p)}{p - 1})$ for $0 \leq i \leq g$ in the aforementioned context where the axes $\Lambda_{(i)}$ are defined.

Finally, the following elementary facts relating valuations to distances from the axis $\Lambda_{0, \infty}$, which are in \cite[Lemma 4.5]{yelton2024branch}, will be used many times in arguments below, and so for convenience we present them here in full.

\begin{prop} \label{prop valuations}

In addition to the above notation, for any $r \in \rr$, write $D(r) \subset \cc_K$ for the disc $\{z \in \cc_K \ | \ v(z) \geq r\}$.

\begin{enumerate}[(a)]

\item Given any point $a \in \proj_{\cc_K}^1 \smallsetminus \{0, \infty\}$, the closest point in the axis $\Lambda_{0, \infty}$ to $\eta_a$ is $\eta_{D(v(a))}$.

\item Given any distinct points $\eta, \eta' \in \Berk$, the point $\eta \vee \eta' \in [\eta, \eta']$ has minimal distance to the axis $\Lambda_{0, \infty}$ among points in the path $[\eta, \eta']$.

\item Given any distinct points $a, b \in \proj_{\cc_K}^1 \smallsetminus \{0, \infty\}$, writing $r = \min\{v(a), v(b)\}$, we have 
\begin{equation} \label{eq valuations}
v(a - b) = r + \delta(\eta_a \vee \eta_b, D(r)) = r + \delta(\Lambda_{a, b}, \Lambda_{0, \infty}).
\end{equation}

\end{enumerate}

\end{prop}
 
\subsection{Convex hulls of finite sets} \label{sec berk convex hulls}

By studying a finite subset $A \subset \proj_K^1$ through its \emph{convex hull} in $\Berk$, we adopt a more topological point of view which enables us to assert stronger results and to better demonstrate them.  For our purposes, we will only care about the convex hull of a set $A$ which is \emph{clustered in pairs}, so we begin with that definition.

\begin{dfn} \label{dfn clustered in pairs}

Let $A \subset \proj_K^1$ be a non-empty even-cardinality subset.  We say that $A$ is \emph{clustered in $r$-separated pairs} for some $r \in \rr_{\geq 0}$ if there is a labeling $a_0, b_0, \dots, a_g, b_g$ of the elements of $A$ such that we have $B(\Lambda_{(i)}, r) \cap B(\Lambda_{(j)}, r) = \varnothing$ for $i \neq j$.

If $A$ is clustered in $0$-separated pairs, we say more simply that $A$ is \emph{clustered in pairs}.

In the context of using this terminology, the \emph{pairs} that $A$ is clustered in are the $2$-element sets $\{a_i, b_i\}$.

\end{dfn}

The following crucial fact comes from \cite[Remark 2.15]{yelton2024branch}.

\begin{prop} \label{prop clustered in pairs}

Every $p$-superelliptic set $S$ is clustered in $\frac{v(p)}{p - 1}$-separated pairs $\{a_0, b_0\}, \dots, \{a_g, b_g\}$; in other words, we have $\hat{\Lambda}_{(i)} \cap \hat{\Lambda}_{(j)} = \varnothing$ for $i \neq j$.  Moreover, the partition $S = \bigsqcup_{i = 0}^g \{a_i, b_i\}$ into pairs is the only one with respect to which $S$ is clustered in pairs.

\end{prop}

In light of the above proposition, in the context of this paper, all subsets $S \subset \proj_K^1$ consisting of fixed points of generators of $p$-Whittaker groups satisfy this property of being clustered in $\frac{v(p)}{p - 1}$-separated pairs (which, in the case that $p$ is not the residue characteristic of $K$, simply means being clustered in pairs).

\begin{dfn} \label{dfn convex hull}

Let $A \subset \proj_K^1$ be a subset which is clustered in the pairs $\{a_0, b_0\}, \dots, \{a_g, b_g\}$ for some $g \geq 0$.  We denote by $\Sigma_A \subset \Berk$ the \emph{convex hull} of $A$, \textit{i.e.} the smallest connected subspace of $\Berk$ containing the set of points of Type I corresponding to $A$.

A \emph{distinguished vertex} of $\Sigma_A$ is a point $v \in \Lambda_{a_i, b_i} \subset \Sigma_A$ for some index $i \in \{0, \dots, g\}$ satisfying that no neighborhood of $v$ in the metric space $\Sigma_A$ is contained in the axis $\Lambda_{(i)}$.

A \emph{natural vertex} of $\Sigma_A$ is a point $v \in \Sigma_A$ whose open neighborhoods contain a star shape centered at $v$ (with $\geq 3$ edges coming out of $v$).

A \emph{vertex} of $\Sigma_A$ is a distinguished or a natural vertex.

\end{dfn}

\begin{rmk} \label{rmk vertices}

The idea behind the above terminology is that the closure of the subspace $\Sigma_A \smallsetminus \{\Lambda_{(i)}\}_{0 \leq i \leq g} \subset \Sigma_A$ is a finite metric graph whose vertices are precisely the vertices of $\Sigma_A$ as described in \Cref{dfn convex hull}.  See \cite[\S3.1]{yelton2024branch} for a proof and more details.

\end{rmk}

\begin{prop} \label{prop clustered in pairs distinguished vertices}

Let $A \subset \proj_K^1$ be a subset which is clustered in pairs.  The set $A$ is clustered in $r$-separated pairs for some $r \in \rr_{> 0}$ if and only if for each pair of distinguished vertices $v, v' \in \Sigma_A$ which do not lie in the same axis $\Lambda_{(i)}$ for any index $i$, we have $\delta(v, v') > 2r$.

\end{prop}

\begin{proof}

We begin by observing that being clustered in $r$-separated pairs is clearly equivalent to the condition that $\delta(\Lambda_{(i)}, \Lambda_{(j)}) > 2r$ for any indices $i \neq j$; we freely use this fact below.

Let $v, v' \in \Sigma_A$ be distinguished vertices not lying in the same axis $\Lambda_{(i)}$; as distinguished vertices by definition each lie in an axis of this type, there are indices $i \neq j$ such that $v \in \Lambda_{(i)}$ and $v' \in \Lambda_{(j)}$.  The forward direction of the assertion now directly follows.  The other direction follows from the fact that, immediately from the definition of distinguished vertex, given any indices $i \neq j$, the closest point in $\Lambda_{(i)}$ (resp. $\Lambda_{(j)}$) to $\Lambda_{(j)}$ (resp. $\Lambda_{(i)}$) is a distinguished vertex $v$ (resp. $v'$), so $\delta(v, v') > 2r$ implies $\delta(\Lambda_{(i)}, \Lambda_{(j)}) > 2r$.
\end{proof}

\begin{prop} \label{prop vertex cluster correspondence}

Let $A \subset \proj_K^1$ be a subset which is clustered in the pairs $\{a_0, b_0\}, \dots, \{a_g, b_g\}$ for some $g \geq 0$.

\begin{enumerate}[(a)]

\item Define an equivalence relation $\sim$ on $A$ as follows: given two points $z, w \in A$, we write $z \sim w$ if $z$ and $w$ lie in the exact same even-cardinality clusters of $A$, \textit{i.e.} if for every even-cardinality cluster $\mathfrak{s} \subseteq A$ we have either $z, w \in \mathfrak{s}$ or $z, w \notin \mathfrak{s}$.  Then the equivalence classes under the relation $\sim$ are precisely the pairs $\{z_i, w_i\}$.

\item There is a one-to-one correspondence between the clusters $\mathfrak{s}$ of $A$ satisfying $2 \leq \#\mathfrak{s} \leq 2g + 1$ and the vertices of $\Sigma_A$ given by sending a cluster $\mathfrak{s}$ to the point $\eta_{D_{\mathfrak{s}}}$, where $D_{\mathfrak{s}} \subset \cc_K$ is the smallest disc containing $\mathfrak{s}$.  Under this correspondence, a cluster which is not (resp. is) itself the union of $\geq 2$ even-cardinality sub-clusters gets sent to a distinguished vertex (resp. a non-distinguished vertex).

\end{enumerate}

\end{prop}

\begin{proof}

Throughout this proof, for any index $i$, we write $\mathfrak{s}_i \subset A$ for the minimal cluster containing the elements $a_i, b_i$, and we write $D_i \subset \cc_K$ for the corresponding disc; note that we have $\eta_{D_i} = \eta_{a_i} \vee \eta_{b_i}$.  Fix an index $i$, and suppose that there is an even-cardinality cluster $\mathfrak{s}$ such that $\#(\mathfrak{s} \cap \{a_i, b_i\}) = 1$.  By considering the cardinality, it is clear that there is an index $l \neq i$ such that $\#(\mathfrak{s} \cap \{a_l, b_l\}) = 1$.  We then clearly have $\eta_{D_{\mathfrak{s}}} \in \Lambda_{(i)} \cap \Lambda_{(l)}$, which contradicts the fact that $S$ is clustered in pairs.  It follows that we have $a_i \sim b_i$.

Now let $i \neq j$ be distinct indices.  We must have $\mathfrak{s}_i \neq \mathfrak{s}_j$, because otherwise we would get $\eta_{a_i} \vee \eta_{b_i} = \eta_{D_i} = \eta_{D_j} = \eta_{a_j} \vee \eta_{b_j} \in \Lambda_{(i)} \cap \Lambda_{(j)}$, which would contradict being clustered in pairs.  Now assume without loss of generality that $\mathfrak{s}_i$ does not contain $\mathfrak{s}_j$.  This means by definition of $\mathfrak{s}_j$ that we have $a_j, b_j \notin \mathfrak{s}_i$, which is also the case if $\mathfrak{s}_i$ and $\mathfrak{s}_j$ are disjoint.  In order to prove that $a_j, b_j \not\sim a_i, b_i$, it then suffices to show that both $\mathfrak{s}_i$ and $\mathfrak{s}_j$ have even cardinality.  If the cluster $\mathfrak{s}_i$ had odd cardinality, that would imply that there is an index $l$ such that $\#(\mathfrak{s}_i \cap \{a_l, b_l\}) = 1$ and therefore that $\eta_{D_i} = \eta_{a_i} \vee \eta_{b_i}$ lies in $\Lambda_{(l)}$; since $\eta_{a_i} \vee \eta_{b_i}$ also lies in $\Lambda_{(i)}$, this again contradicts the fact that $S$ is clustered in pairs.  Therefore, we have $\mathfrak{s}_i$ has even cardinality and, by the exact same argument, so does $\mathfrak{s}_j$.  This completes the proof of part (a).

Let $\mathfrak{s}$ be a cluster of $A$ which is not the union of $\geq 2$ even-cardinality sub-clusters.  Choose an element in $A$ which does not lie in a proper even-cardinality sub-cluster of $A$; we may call this element $a_i$ for some index $i \in \{0, \dots, g\}$.  Then since we have $a_i \sim b_i$ by part (a), we either have $b_i \notin A$ or that $b_i$ is also an element in $A$ which does not lie in a proper even-cardinality sub-cluster of $A$.  In the former case, we immediately get $\eta_{D_{\mathfrak{s}}} \in \Lambda_{(i)}$, whereas in the latter case, we claim that $\eta_{D_{\mathfrak{s}}} = \eta_{a_i} \vee \eta_{b_i} \in \Lambda_{(i)}$.  To see this, let $\mathfrak{c} \subseteq \mathfrak{s}$ be the smallest sub-cluster containing $a_i, b_i$.  If there were $\geq 3$ elements of $\mathfrak{c}$ not lying in a proper even-cardinality sub-cluster of $\mathfrak{c}$, they would all lie in the same equivalence class and contradict part (a), so it must be the case that $a_i, b_i \in \mathfrak{c}$ are the only elements not lying in a proper even-cardinality sub-cluster, and so $\mathfrak{c}$ has even cardinality.  By construction of $a_i, b_i$, we then get $\mathfrak{c} = \mathfrak{s}$, which implies our claim.  We therefore have $\eta_{D_{\mathfrak{s}}} \in \Lambda_{(i)}$, and we now only need to show that any neighborhood of $\eta_{D_{\mathfrak{s}}}$ contains a point $\eta \in \Sigma_S \smallsetminus \Lambda_{(i)}$, as making such a neighborhood small enough ensures that $\eta$ does not lie in any other axis $\Lambda_{(j)}$.  Note that for any index $j \neq i$, we have $[\eta_{D_{\mathfrak{s}}}, \eta_{a_j}] \subset \Sigma_S$ and that every neighborhood of $\eta_{D_{\mathfrak{s}}}$ contains a sub-path $[\eta_{D_{\mathfrak{s}}}, \eta_{D_\epsilon}] \subset [\eta_{D_{\mathfrak{s}}}, \eta_{a_j}]$.  If the cluster $\mathfrak{s}$ has even cardinality, then there is an index $j \neq i$ such that $a_j \notin \mathfrak{s}$, and then one may take $D_\epsilon$ to be a slightly larger disc containing $D_{\mathfrak{s}} = D_i$, and it is clear that $D_\epsilon \notin \Lambda_{(i)}$.  If, on the other hand, the cluster $\mathfrak{s}$ has odd cardinality, then we have $a_i \in \mathfrak{s}$ and $b_i \notin \mathfrak{s}$; as the elements of $\mathfrak{s} \smallsetminus \{a_i\}$ are not equivalent to $a_i$, there is an even-cardinality sub-cluster $\mathfrak{c} \subseteq \mathfrak{s} \smallsetminus \{a_i\}$.  Choosing $a_j \in \mathfrak{c}$, we get that $a_i, b_i \notin D_\epsilon \subsetneq D_{\mathfrak{s}}$, so that $\eta_{D_\epsilon} \notin \Lambda_{(i)}$.  So in either case, we are done showing that the point $\eta_{D_{\mathfrak{s}}}$ corresponding to the cluster $\mathfrak{s}$ is a distinguished vertex and have proved one case of part (b).

Now let $\mathfrak{s}$ be a cluster of $A$ which is the union of even-cardinality sub-clusters; let us write $\mathfrak{s} = \mathfrak{c}_1 \sqcup \dots \sqcup \mathfrak{c}_s$ for some $s \geq 2$, where each $\mathfrak{c}_l$ is a maximal proper even-cardinality sub-cluster of $\mathfrak{s}$.  For each index $i$, it follows from part (a) that we either have $a_i, b_i \in \mathfrak{c}_l$ for some $l$ or have $a_i, b_i \notin \mathfrak{s}$.  In the former case, we have $D_{\mathfrak{s}} \supsetneq D_i$, and in the latter case, we have $a_i, b_i \notin D_{\mathfrak{s}}$, so in either case we get $\eta_{D_{\mathfrak{s}}} \notin \Lambda_{(i)}$, implying that $\eta_{D_{\mathfrak{s}}}$ is not a distinguished vertex.  For $1 \leq l \leq s$, choose an element $a_{(l)} \in \mathfrak{c}_l$.  We have $[\eta_{D_{\mathfrak{s}}}, \eta_{a_{(l)}}] \subset \Sigma_S$ and that every neighborhood of $\eta_{D_{\mathfrak{s}}}$ contains a sub-path $[\eta_{D_{\mathfrak{s}}}, \eta_{D_{\epsilon, l}}] \subset [\eta_{D_{\mathfrak{s}}}, \eta_{a_{(l)}}]$, with $\mathfrak{c}_l \subset D_{\epsilon, l} \subsetneq D_{\mathfrak{s}}$.  At the same time, as there is some index $i$ such that $a_i \notin \mathfrak{s}$, we have $[\eta_{D_{\mathfrak{s}}}, \eta_{a_i}] \subset \Sigma_S$ and that every neighborhood of $\eta_{D_{\mathfrak{s}}}$ contains a sub-path $[\eta_{D_{\mathfrak{s}}}, \eta_{D_{\epsilon, 0}}] \subset [\eta_{D_{\mathfrak{s}}}, \eta_{a_{(l)}}]$, with $D_{\epsilon, 0} \supsetneq D_{\mathfrak{s}}$.  It is clear that the shortest path between any pair of these points $\eta_{D_{\epsilon, 0}}, \eta_{D_{\epsilon, 1}}, \dots, \eta_{D_{\epsilon, s}}$ passes through $\eta_{D_{\mathfrak{s}}}$ and therefore, the neighborhood contains a star shape centered at $\eta_{D_{\mathfrak{s}}}$ (with at least $s + 1 \geq 3$ edges coming out).  The point $\eta_{D_{\mathfrak{s}}}$ is thus a natural vertex.  This completes the proof of part (b).
\end{proof}

\subsection{Cluster data and metric properties of convex hulls} \label{sec berk dictionary}

We may now establish a dictionary between the cluster data of a finite subset $A \subset \proj_K^1$ and the metric properties of its convex hull $\Sigma_S$.  In order to express the following results, we define an order relation (denoted by $>$) on $\Berk$ by decreeing that $\xi > \xi'$ for any points $\xi, \xi' \in \Berk$ means that $\xi \in [\xi', \eta_\infty] \smallsetminus \{\xi'\}$.  Note that by this definition, we have $\eta_D > \eta_{D'}$ for discs $D, D' \subset \cc_K$ if and only if we have $D \supsetneq D'$.

The following proposition justifies the use of the letter $\delta$ for both depth of a cluster and the distance function on $\Berk$.

\begin{prop} \label{prop depths distances}

Suppose that $\mathfrak{s}$ is a non-maximal cluster of a set $S$ which is clustered in pairs, and let $v = \eta_{D_{\mathfrak{s}}} \in \Sigma_S$ be the vertex corresponding to $\mathfrak{s}$ as in \Cref{prop vertex cluster correspondence}(b).  Then the relative depth $\delta(\mathfrak{s})$ is equal to $\delta(v, v')$, where $v' \in \Sigma_S$ is the closest vertex to $v$ satisfying $v' > v$.

\end{prop}

\begin{proof}

It is clear from definitions and from the correspondence between clusters and vertices established by \Cref{prop vertex cluster correspondence}(b) that the closest vertex $v'$ to $v$ satisfying $v' > v$ corresponds to the smallest cluster $\mathfrak{s}'$ which properly contains $\mathfrak{s}$.  We have $\delta(\mathfrak{s}) = d(\mathfrak{s}) - d(\mathfrak{s}')$, which equals the difference in logarithmic radii of the minimal discs $D_{\mathfrak{s}}, D_{\mathfrak{s}'}$ respectively containing the clusters $\mathfrak{s}, \mathfrak{s}'$, and which in turn by definition equals the distance between the respective points $v = \eta_{D_{\mathfrak{s}}}, v' = \eta_{D_{\mathfrak{s}'}}$ in $\Berk$.
\end{proof}

Our above results now allow us to present a definition in the language of clusters of being clustered in $r$-separated pairs.

\begin{prop} \label{prop clustered in pairs alternate definition}

Given any cluster $\mathfrak{s} \subset A$, write $\mathfrak{s}'$ (resp. $\mathfrak{s}^\sim$) for the smallest cluster properly containing $\mathfrak{s}$ (resp. properly containing $\mathfrak{s}$ and which is not itself the disjoint union of $\geq 2$ even-cardinality sub-clusters), if such a cluster properly containing $\mathfrak{s}$ exists.

A set $A$ is clustered in $r$-separated pairs for some $r \in \rr_{> 0}$ if and only if the following conditions hold, for any even-cardinality clusters $\mathfrak{c}, \mathfrak{c}_1, \mathfrak{c}_2$ which are themselves not the disjoint union of $\geq 2$ even-cardinality clusters:
\begin{enumerate}[(i)]
\item the set $A$ is clustered in pairs;
\item if $\mathfrak{c}^\sim$ is defined, we have $d(\mathfrak{c}) - d(\mathfrak{c}^\sim) > 2r$; and 
\item if $\mathfrak{c}_1', \mathfrak{c}_2'$ are defined and $\mathfrak{s} := \mathfrak{c}_1' = \mathfrak{c}_2'$, we have $\delta(\mathfrak{c}_1) + \delta(\mathfrak{c}_2) > 2r$. 
\end{enumerate}

\end{prop}

\begin{proof}

We know from \Cref{prop clustered in pairs distinguished vertices} that $A$ is clustered in $r$-separated pairs if and only if, for each pair of distinguished vertices $v, v' \in \Sigma_A$ which do not lie on the same axis $\Lambda_{(i)}$, we have $\delta(v, v') > 2r$; we will show that this latter condition is equivalent to properties (i)-(iii).

Property (i) is implied by the property of being clustered in $r$-separated pairs by definition.  Note that for any point $w = \eta_D$ sufficiently close to the vertex $v = \eta_{D_{\mathfrak{s}}}$ corresponding (as in the statement of \Cref{prop vertex cluster correspondence}) to an even-cardinality cluster $\mathfrak{s}$ with $w > v$, we have $D \cap A = \mathfrak{s}$ and $D \supsetneq D_{\mathfrak{s}}$ so that $w = \eta_D$ does not lie in any axis $\Lambda_{(i)}$.  Thus, if $v'$ is another vertex such that the path $[v, v'] \subset \Sigma_A$ contains a point $w > v$, it is not possible for $v$ and $v'$ to lie in the same axis $\Lambda_{(i)}$.  Now using \Cref{prop vertex cluster correspondence}(b) and \Cref{prop depths distances}, properties (ii) and (iii) can each be interpreted as saying that the distance between a pair of distinguished vertices satisfying a certain property is $> 2r$ (consulting \Cref{dfn berk} and keeping in mind $\delta(\mathfrak{c}_1) + \delta(\mathfrak{c}_2) = d(\mathfrak{c}_1) + d(\mathfrak{c}_2) - 2d(\mathfrak{s})$ and that $d(\mathfrak{c}) = d(D_{\mathfrak{c}})$ for any cluster $\mathfrak{c}$).  In the case of property (ii), these distinguished vertices $v, v'$ satisfy that $v'$ is the closest distinguished vertex to $v$ such that $v' > v$, while in the case of property (iii), these distinguished vertices $w_1, w_2$ satisfy that $v := w_1 \vee w_2$ is the closest vertex to $w_1$ (resp. $w_2$) such that $v > w_1$ (resp. $v > w_2$).  In both cases, our above observations imply that the two distinguished vertices in question cannot lie in the same axis $\Lambda_{(i)}$.  The property of being clustered in $r$-separated pairs therefore implies properties (ii) and (iii).

Conversely, suppose that a set $A$ which is clustered in pairs satisfies properties (ii) and (iii), and choose distinguished vertices $v_1 \neq v_2$ of $\Sigma_A$; we need to prove that $\delta(v_1, v_2) > 2r$.  It clearly suffices to replace $v_2$ with the closest distinguished vertex to $v_1$ in the half-open segment $[v_1, v_2] \smallsetminus \{v_1\}$.  Having made such a replacement, if we have $v_2 > v_1$ or $v_1 > v_2$, then the desired inequality is provided by property (ii).  If we do not have $v_2 > v_1$ or $v_1 > v_2$, then it is clear that $v_1 \vee v_2 > v_1, v_2$ satisfies that there is no distinguished vertex lying in the interior of either of the paths $[v_1, v_1 \vee v_2], [v_2, v_1 \vee v_2]$, and so the desired inequality is provided by property (iii).
\end{proof}

\begin{rmk} \label{rmk cluster data and graph}

Given a subset $A \subset \proj_K^1$ which is clustered in pairs, let $\Sigma_A^* \subset \Sigma_A$ be the convex hull of the vertices of $\Sigma_A$ (\textit{i.e.} the smallest connected subspace of $\Sigma_A$ containing its vertices).  Propositions \ref{prop vertex cluster correspondence}(b) and \ref{prop depths distances} show that $\Sigma_A^*$ is a finite metric graph and that the cluster data (resp. the combinatorial cluster data) of $A$ determines and is determined (up to alterations induced by replacing $A$ with its image under a fractional linear transformation) by the metric graph isomorphism type (resp. the graph isomorphism type) of $A$.  More precisely, if we let $\varphi: A \to A'$ be a bijection of finite subsets of $\proj_K^1$ and assume that we have $\infty \in A \cap A'$ and that $\varphi(\infty) = \infty$, then the map $\varphi$ acts as a bijection between clusters of $A$ and clusters of $A'$ if and only if this bijection of clusters, viewed as a bijection between the vertex sets of $\Sigma_A^*$ and $\Sigma_{A'}^*$ via \Cref{prop vertex cluster correspondence}(b), is a graph isomorphism; moreover, this bijection preserves the relative depths of clusters if and only this corresponding graph isomorphism is an isometry.

\end{rmk}

\begin{prop} \label{prop berkovich version implies cluster version}

\Cref{thm main berk} implies \Cref{thm main clusters}.

\end{prop}

\begin{proof}

\Cref{thm main berk} says that, under the optimality hypothesis which is common to both theorems' statements, the bijection $\pi : S \to \mathcal{B}$ extends to a map $\pi_* : \Sigma_S \to \Sigma_{\mathcal{B}}$ between the convex hulls which is a homeomorphism.  This last property implies that the image of the path between two points of Type I under $\pi_*$ coincides with the path between the images of those two points of Type I, so that in particular, we have $\pi_*(\Lambda_{a_i, b_i}) = \Lambda_{\alpha_i, \beta_i}$ for $0 \leq i \leq g$.  Now from the definitions and the fact that $\pi_*$ is a homeomorphism, it follows that $\pi_*$ acts as a bijection between distinguished (resp. natural) vertices of $\Sigma_S$ and those of $\Sigma_{\mathcal{B}}$.  By \Cref{prop vertex cluster correspondence}(b), then the bijection $\pi$ preserves combinatorial cluster data, thus implying \Cref{thm main clusters}(a).

Suppose that $\mathfrak{s}$ is an odd-cardinality cluster; let $v \in \Sigma_S$ be the corresponding vertex as in \Cref{prop vertex cluster correspondence}(b); and let $v' \in \Sigma_S$ be the closest vertex to $v$ satisfying $v' > v$.  Then the disc $D$ corresponding to any point in the interior of the path $[v, v'] \subset \Sigma_S$ must satisfy $D \cap A = \mathfrak{s}$ and therefore its intersection with $\{a_i, b_i\} \subset A$ must be a singleton for some index $i$; it is easy to deduce from this that we must then have $[v, v'] \subset \Lambda_{(i)}$.  Then, in the notation of \Cref{thm main berk}, we have $\llbracket v, v' \rrbracket = [v, v']$, implying $\mu(v, v') = \delta(v, v')$, and so the formula in (\ref{eq dilation}) gives us 
$\delta(\pi_*(v), \pi_*(v')) = p\delta(v, v')$.  Now we may apply \Cref{prop depths distances} to get \Cref{thm main clusters}(b)(i).

Now suppose that $\mathfrak{s}$ is an even-cardinality cluster, and define $v, v' \in \Sigma_S$ as before.  Again, the disc $D$ corresponding to any point in the interior of the path $[v, v'] \subset \Sigma_S$ satisfies $D \cap A = \mathfrak{s}$, and it is not the smallest such disc; by the hypothesis of being clustered in pairs, for each index $i$, we have $D \cap \{a_i, b_i\} = \varnothing$ or $D \supset \{a_i, b_i\}$, and so no point in the interior of $[v, v']$ lies in any axis $\Lambda_{(i)}$.  Assume for the moment that $p$ is not the residue characteristic of $K$.  Then, in the notation of \Cref{thm main berk}, we have $\mu(v, v') = 0$, and therefore, the formula in (\ref{eq dilation}) gives us $\delta(\pi_*(v), \pi_*(v')) = \delta(v, v')$, and applying \Cref{prop depths distances} implies \Cref{thm main clusters}(b)(ii).  Now let us drop any assumption on $p$ and rather assume for the moment that neither $\mathfrak{s}$ nor $\mathfrak{s}'$ is the union of $\geq 2$ even-cardinality clusters.  Then by \Cref{prop vertex cluster correspondence}(b), the corresponding vertices $v, v' \in \Sigma_S$ are distinguished.  Since there is no distinguished vertex in the interior of the path $[v, v']$ (as this interior does not intersect any axis $\Lambda_{(i)}$), we get 
\begin{equation}
\llbracket v, v' \rrbracket = [v, \tilde{v}] \sqcup [\tilde{v}', v'],
\end{equation}
where $\tilde{v}$ (resp. $\tilde{v}'$) is the (unique) point in $[v, v']$ of distance $\frac{v(p)}{p - 1}$ from $v$ (resp. $v'$).  We thus have $\mu(v, v') = \frac{2v(p)}{p - 1}$.  Therefore, the formula in (\ref{eq dilation}) gives us $\delta(\pi_*(v), \pi_*(v') = \delta(v, v') + 2v(p)$, and applying \Cref{prop depths distances} implies \Cref{thm main clusters}(b)(iii).
\end{proof}

\section{Proof of the main theorem} \label{sec proof of main}

This section is devoted to proving \Cref{thm main berk}.  Our first task is to gather some background results on the action of $\PGL_2(K)$ on the convex hull $\Sigma_S$ of an optimal subset $S \subset \proj_K^1$ which are variants of (and are proved using) results in the author's previous paper \cite{yelton2024branch}; this is done in \S\ref{sec proof of main p-Whittaker}.  Our method of proving our main result requires making explicit the modulo-action-of-$\Gamma_0$ function that takes $\Omega$ to $\proj_K^1 \cong \Omega / \Gamma_0$ as a \emph{theta function}, which is the topic of \S\ref{sec proof of main theta}.  The remaining three subsections are then dedicated to the actual proof, which is broken into three parts: first (in \S\ref{sec proof of main approximation}) the presentation and proof of a result (\Cref{thm one segment}) providing an approximation of values of one of these theta functions $\Theta$ at certain inputs under a simplifying hypothesis on $S$; then (in \S\ref{sec proof of main one segment}) the construction of an extension $\Theta_*$ of the theta function on $S$ to a map from the convex hull $\Sigma_S$ to $\Berk$ and an explicit formula for $\Theta_*$ when restricted to a single segment of $\Sigma_S$ (with the simplifying hypothesis retained); and finally (in \S\ref{sec proof of main whole}) a ``gluing'' argument that completes the proof of \Cref{thm main berk}.

We return to the setting where we have a $p$-superelliptic subset $S \subset \proj_K^1$, meaning that (in particular) the set $S$ is clustered in $\frac{v(p)}{p - 1}$-separated pairs $\{a_0, b_0\}, \dots, \{a_g, b_g\}$, and that, letting $s_i \in \PGL_2(K)$ be an order-$p$ automorphism fixing $a_i, b_i \in S$ for $0 \leq i \leq g$, the group $\Gamma_0 := \langle s_0, \dots, s_g \rangle$ is a $p$-Whittaker group (and so in particular is isomorphic to the free product of its cyclic subgroups $\langle s_i \rangle$).  The Schottky group $\Gamma \lhd \Gamma_0$ is the index-$p$ normal subgroup consisting of words on the generators $s_i$ whose total exponent is divisible by $p$.  Throughout this section, let $\Omega = \Omega_\Gamma = \Omega_{\Gamma_0}$ denote the set of non-limit points of $\Gamma$ (and of $\Gamma_0$) in $\proj_{\cc_K}^1$, noting that in previous parts of the paper we wrote $\Omega$ to refer to the set non-limit $K$-points rather than the non-limit $\cc_K$-points.

\subsection{Some useful results on $p$-Whittaker groups} \label{sec proof of main p-Whittaker}

Before we can start to prove \Cref{thm main berk} in earnest, we need to establish some useful properties of the action of the $p$-Whittaker group $\Gamma_0$ on the convex hull $\Sigma_S$ in the case that $S$ is optimal.  We begin by presenting a variation of our previous result \cite[Lemma 3.16]{yelton2024branch}.

\begin{lemma} \label{lemma off the convex hull variant}

Let $S \subset \proj_K^1$ be an optimal subset with associated Schottky group $\Gamma$, and choose a point $v \in \Sigma_S$ and a nontrivial element $\gamma \in \Gamma$, which we write as a word 
\begin{equation} \label{eq word}
\gamma = s_{i_t}^{n_t} s_{i_{t - 1}}^{n_{t - 1}} \cdots s_{i_1}^{n_1}
\end{equation}
 for some $t \geq 1$, some $n_1, \dots, n_t \in \zz \smallsetminus p\zz$, and some indices $i_l$ satisfying $i_l \neq i_{l - 1}$ for $2 \leq l \leq t$.  Then the closest point in $\Sigma_S$ to $\gamma(\eta)$ lies in $\hat{\Lambda}_{i_t}$, and in fact we have 
\begin{equation} \label{eq distance from gamma(eta) to hull}
\delta(\gamma(\eta), \Sigma_S) \geq \delta(\eta, \hat{\Lambda}_{i_1}) + \sum_{l = 2}^t \delta(\hat{\Lambda}_{i_{l-1}}, \hat{\Lambda}_{i_l}) > 0.
\end{equation}

\end{lemma}

\begin{proof}

If we have $v \notin \hat{\Lambda}_{i_1}$, then the hypothesis of \cite[Lemma 3.16]{yelton2024branch} applies, and that result gives us the desired conclusion.  If instead we have $v \in \hat{\Lambda}_{i_1}$, then the element $\gamma s_{t_1}^{-n_1} \in \Gamma_0$ is nontrivial, and the hypothesis of \cite[Lemma 3.16]{yelton2024branch} applies when replacing $\gamma$ with $\gamma s_{t_1}^{-n_1}$.  Keeping in mind that $\delta(v, \hat{\Lambda}_{i_1}) = 0$ in this case, this again gives us the desired conclusion.
\end{proof}

\begin{lemma} \label{lemma distance gamma(b) a}

Let $S \subset \proj_K^1$ be an optimal subset; let $a \in \proj_K^1$ be any point such that the closest point in $\Sigma_S$ to $\eta_a$ is not a distinguished vertex; and let $b \in \proj_K^1$ be any point.  Then there exists $M \in \qq$ such that we have $v(\gamma(b) - a) > M$ for at most one element $\gamma \in \Gamma$.

\end{lemma}

\begin{proof}

We first set out to show that there is at most one element $\gamma \in \Gamma$ such that the closest point in $\Sigma_S$ to $\eta_{\gamma(b)}$ is not a distinguished vertex.  Suppose that there exists an element $\gamma_0 \in \Gamma$ such that the closest point $\xi$ in $\Sigma_S$ to $\eta_{\gamma_0(b)}$ is not a distinguished vertex.  Choose a nontrivial element $\gamma \in \Gamma$ and assume that the path $[\eta_{\gamma\gamma_0(b)} = \gamma(\eta_{\gamma_0(b)}), \gamma(\xi)] \subset \Berk$ intersects $\Sigma_S$ at a point $\xi' \in \Sigma_S$.  \Cref{lemma off the convex hull variant} implies that $\xi' \neq \gamma(\xi)$, so $\xi'$ lies in the interior of the path $[\gamma(\eta_{\gamma_0(b)}), \gamma(\xi)]$.  Applying the action of $\gamma^{-1}$ shows us that $\gamma^{-1}(\xi')$ lies in the interior of the path $[\eta_{\gamma_0(b)}, \xi]$.  By \Cref{lemma off the convex hull variant}, the closest point in $\Sigma_S$ to $\gamma^{-1}(\xi')$ is a distinguished vertex, which contradicts the fact that the closest point in $\Sigma_S$ to $\eta_{\gamma_0(b)}$ (and thus also to $\gamma^{-1}(\xi')$) is not a distinguished vertex.  From this contradiction we get $[\eta_{\gamma \gamma_0(b)}, \gamma(\xi)] \cap \Sigma_S = \varnothing$.  By \Cref{lemma off the convex hull variant}, the closest point in $\Sigma_S$ to $\gamma(\xi)$ is a distinguished vertex; it follows that the closest point in $\Sigma_S$ to $\eta_{\gamma \gamma_0(b)}$ is a distinguished vertex.  Since $\gamma$ was chosen arbitrarily, we have proved our claim.

We now consider the possible values of $v(\gamma(b) - a)$ over all elements $\gamma \in \Gamma$.  After discarding at most one choice of $\gamma$, we may assume that the closest point in $\Sigma_S$ to $\eta_{\gamma(b)}$ is a distinguished vertex.  Since the closest point in $\Sigma_S$ to $\eta_a$ is not a distinguished vertex, we have $\Lambda_{\gamma(b), a} \cap \Sigma_S \neq \varnothing$ and that this intersection is the path $[\xi, \xi'_\gamma]$, where $\xi$ (resp. $\xi'_\gamma$) is the closest point in $\Sigma_S$ to $\eta_a$ and (resp. $\eta_{\gamma(b)}$).  In terms of the partial order established at the top of \S\ref{sec berk dictionary}, we get either $\eta_a \vee \eta_{\gamma(b)} \in \{\xi, \xi'_\gamma\}$ or $\eta_a \vee \eta_{\gamma(b)} > \xi, \xi'_\gamma$.  It follows that $\eta_{\gamma(b)} \vee \eta_a$ is the point $\xi \in \Sigma_S$, is a distinguished vertex of $\Sigma_{S, 0}$, or is greater than a distinguished vertex $v \in \Sigma_{S, 0}$ with respect to our partial order; the last condition implies that $\eta_a \vee \eta_{\gamma(b)} \in [v, \eta_\infty]$ and so we get $\delta(\eta_a \vee \eta_{\gamma(b)}, \Lambda_{0, \infty}) \leq \delta(v, \Lambda_{0, \infty})$.  Thus, the distance $\delta(\eta_a \vee \eta_{\gamma(b)}, \Lambda_{0, \infty})$ either equals $\delta(\xi, \Lambda_{0, \infty})$ or is at most the maximum distance between a distinguished vertex and the axis $\Lambda_{0, \infty}$.  Now applying \Cref{prop valuations}(c) yields the inequality 
\begin{equation}
v(\gamma(b) - a) \leq M := v(a) + \max(\{\delta(\xi, \Lambda_{0, \infty})\} \cup \{\delta(v, \Lambda_{0, \infty})\}_v),
\end{equation}
where the maximum is taken over all distinguished vertices $v$ (of which there are only finitely many).  Since our choice of $\gamma \in \Gamma$ was arbitrary apart from possibly discarding one element, the assertion of the lemma follows.
\end{proof}

\begin{cor} \label{cor S in Omega}

Let $S \subset \proj_K^1$ be an optimal subset, and let $a \in \proj_K^1$ be an element satisfying that the closest point in $\Sigma_S$ to $\eta_a$ is not a distinguished vertex.  Then we have $a \in \Omega$.  In particular, we have $S \subset \Omega$.

\end{cor}

\begin{proof}

Suppose that a point $a$ satisfying this hypothesis is a limit point.  Then for some $b \in \proj_K^1$ and some subset $\{\gamma_n\}_{n \geq 1} \subset \Gamma$, we have $\lim_{n \to \infty} \gamma_n(b) = a$, or equivalently, $\lim_{n \to \infty} v(\gamma_n(b) - a) = \infty$.  But this is contradicted by \Cref{lemma distance gamma(b) a}.
\end{proof}

\begin{rmk} \label{rmk S in Omega}

As is explained in \S\ref{sec intro previous results}, any $p$-superelliptic subset $S \subset \proj_K^1$ can be ``folded into'' an optimal set $\Smin$ without affecting the associated $p$-Whittaker group $\Gamma_0$; modifying $S$ via a folding amounts to acting on each element of $S$ by some automorphism in $\Gamma_0$.  Since the set of non-limit points $\Omega$ of $\Gamma_0$ is invariant under the action of $\Gamma_0$, the second statement of \Cref{cor S in Omega}, which says that $\Smin \subset \Omega$, implies that we have $S \subset \Omega$ as well.  This crucial fact about fixed points of generators of $p$-Whittaker groups is proved (using good fundamental domains) in the $p = 2$ case as \cite[Lemma 2.3]{van1983non} and is mentioned at the top of \cite[\S3]{van1982galois} for general $p$; in the latter reference, the author van Steen refers to his thesis for a proof.  Our above argument appears to be more or less independent of van Steen's, and in our context it comes as part of a more general statement which is useful to us in its own right.

\end{rmk}

\begin{lemma} \label{lemma N_gamma}

Let $S \subset \proj_K^1$ be an optimal subset.  For each $\gamma \in \Gamma$ and $0 \leq i \leq g$, write $c_i^\gamma = [\gamma(a_i)]^{-1} \gamma(b_i) - 1$.  Assume that for some index $j$, we have $a_j = 0$ and $b_j = \infty$.

\begin{enumerate}[(a)]

\item For any index $i$ and any $M > 0$ there are only finitely many elements $\gamma \in \Gamma$ such that $v(c_i^\gamma) \leq M$.

\item For any index $i \neq j$, we have 
\begin{equation} \label{eq v(N_gamma)}
v(c_i^\gamma) \geq v(a_i^{-1}b_i - 1) \text{ for all }\gamma \in \Gamma.
\end{equation}
Moreover, under the additional assumption that we have $[\eta_{a_i} \vee \eta_{b_i}, \mathfrak{v}_j] \cap \hat{\Lambda}_{(l)} = \varnothing$ for indices $l \neq i, j$, equality occurs in (\ref{eq v(N_gamma)}) if and only if we have $\gamma \in \langle s_j \rangle$ (that is, only when $\gamma$ acts as $z \mapsto \zeta_p^n z$ for some $n \in \zz$).

\end{enumerate}

\end{lemma}

\begin{proof}

Let $G_j \subset \Gamma$ be the subset of elements which, when written as a word as in (\ref{eq word}), satisfy $i_t \neq j$ (interpreting this definition so that $1 \in G_j$).  It is elementary to check that $\Gamma$ can be written as the disjoint union $\bigsqcup_{n = 0}^{p - 1} s_j^n G_j s_i^{-n}$.  Using the fact that $s_i$ fixes the points $a_i, b_i \in K$, we compute 
\begin{equation}
c_i^{s_j^n \gamma s_i^{-n}} = (\zeta_p^n \gamma(a_i))^{-1} (\zeta_p^n \gamma(b_i)) - 1 = [\gamma(a_i)]^{-1} \gamma(b_i) - 1 = c_i^\gamma.
\end{equation}
It therefore suffices to prove that the statements of parts (a) and (b) hold for $\gamma \in G_j$.  We will show for such $\gamma$ that, letting $t(\gamma)$ denote the \emph{length} of $\gamma$ as a word on the generators $s_l$ of $\Gamma_0$ (that is, $t(\gamma)$ is the natural number $t$ appearing in the expression for $\gamma$ in (\ref{eq word})), we have 
\begin{equation} \label{eq c_i^gamma in terms of t(gamma)}
v(c_i^\gamma) \geq t(\gamma) \min\{\delta(\hat{\Lambda}_{(l)}, \hat{\Lambda}_{(m)})\}_{l \neq m}.
\end{equation}
Since the set $\{\delta(\hat{\Lambda}_{(l)}, \hat{\Lambda}_{(m)})\}_{l \neq m}$ is finite and consists of positive numbers, the inequality (\ref{eq c_i^gamma in terms of t(gamma)}) immediately implies part (a).

We have $c_i^\gamma = a_i^{-1}b_i - 1$ when $\gamma = 1$, which immediately verifies both the inequality in (\ref{eq c_i^gamma in terms of t(gamma)}) and the statement of part (b) in this case, so we choose $\gamma \in G_j \smallsetminus \{1\}$ and proceed to show that both inequalities (\ref{eq c_i^gamma in terms of t(gamma)}) and $v(c_i^\gamma) \geq v(\lambda)$ hold, and that the latter inequality is strict if we have $[\eta_{a_i} \vee \eta_{b_i}, \mathfrak{v}_j] \cap \hat{\Lambda}_{(l)} = \varnothing$ for indices $l \neq i, j$.  Below we will write $t$ for $t(\gamma)$.

To this end, we first note using \Cref{prop valuations}(c) that we have 
\begin{equation} \label{eq first formula for N_i^gamma}
v(c_i^\gamma) = v(\gamma(b_i) - \gamma(a_i)) - v(\gamma(a_i)) = \delta(\eta_{\gamma(a_i)} \vee \eta_{\gamma(b_i)}, \Lambda_{(j)}).
\end{equation}
Now as the fractional linear transformation $\gamma$ acts on $\Berk$ as a self-homeomorphism and sends the endpoints $\eta_{a_i}, \eta_{b_i}$ of the axis $\Lambda_{(i)}$ to the points $\eta_{\gamma(a_i)}, \eta_{\gamma(b_i)}$ respectively, we have $\gamma(\Lambda_{(i)}) = \Lambda_{\gamma(a_i), \gamma(b_i)} \ni \eta_{\gamma(a_i)} \vee \eta_{\gamma(b_i)}$, from which it follows that there is a point $v \in \Lambda_{(i)} \subset \Sigma_S$ such that $\gamma(v) = \eta_{\gamma(a_i)} \vee \eta_{\gamma(b_i)}$.  Now we apply \Cref{lemma off the convex hull variant} to get 
\begin{equation} \label{eq distance from gamma(v) to far axis}
\delta(\gamma(v), \hat{\Lambda}_{(i_t)}) = \delta(v, \hat{\Lambda}_{(i_1)}) + \sum_{l = 2}^t \delta(\hat{\Lambda}_{(i_{l-1})}, \hat{\Lambda}_{(i_l)}) \geq \delta(v, \hat{\Lambda}_{(i_t)}).
\end{equation}
In fact, we may now estimate $v(c_i^\gamma)$ by computing 
\begin{equation} \label{eq estimating v(N_gamma)}
\begin{aligned}
v(c_i^\gamma) = \delta(\gamma(v), \Lambda_{(j)}) &\text{\indent by (\ref{eq first formula for N_i^gamma})} \\
\geq \delta(\gamma(v), \hat{\Lambda}_{(i_t)}) + \delta(\hat{\Lambda}_{(i_t)}, \Lambda_{(j)}) & \text{\indent using \Cref{lemma off the convex hull variant}} \\
\geq \delta(v, \hat{\Lambda}_{(i_t)}) + \delta(\hat{\Lambda}_{(i_t)}, \Lambda_{(j)}) &\text{\indent by (\ref{eq distance from gamma(v) to far axis})} \\
\geq \delta(v, \Lambda_{(j)}) &\text{\indent (strict if $[v, \mathfrak{v}_j] \cap \hat{\Lambda}_{(i_t)} = \varnothing$)} \\
\geq \delta(\Lambda_{(i)}, \Lambda_{(j)}) &\text{\indent because $v \in \Lambda_{(i)}$} \\
= v(b_i - a_i) - v(a_i) = v(a_i^{-1}b_i - 1) &\text{\indent by \Cref{prop valuations}(b)(c),}
\end{aligned}
\end{equation}
where the particular consequence of \Cref{lemma off the convex hull variant} used in the first inequality of (\ref{eq estimating v(N_gamma)} is its assertion that the closest point in $\Sigma_S$ to $\gamma(\mathfrak{v}_i)$ lies in $\hat{\Lambda}_{(i_t)}$.  This directly provides the desired inequality $v(c_i^\gamma) \geq v(a_i^{-1}b_i - 1)$; we see from (\ref{eq estimating v(N_gamma)}) that it is strict if we have $[v, \mathfrak{v}_j] \cap \hat{\Lambda}_{(i_t)} = \varnothing$.  As we have $[v, \mathfrak{v}_j] = [v, \eta_{a_i} \vee \eta_{b_i}] \cup [\eta_{a_i} \vee \eta_{b_i}, \mathfrak{v}_j]$ and $[v, \eta_{a_i} \vee \eta_{b_i}] \subset \Lambda_{(i)}$ is disjoint from $\Lambda_{(j)}$ by the property of being clustered in pairs, the condition that implies strictness is equivalent to the condition that $[\eta_{a_i} \vee \eta_{b_i}, \mathfrak{v}_j] \cap \Lambda_{(i_t)} = \varnothing$.

\enlargethispage{2\baselineskip}

Meanwhile, the sequence of inequalities in (\ref{eq estimating v(N_gamma)}) also includes $v(c_i^\gamma) \geq \delta(\gamma(\mathfrak{v}_i), \hat{\Lambda}_{(i_t)}) + \delta(\hat{\Lambda}_{(i_t)}, \Lambda_{(j)})$ which, using (\ref{eq distance from gamma(v) to far axis}), allows us to get the desired inequality (\ref{eq c_i^gamma in terms of t(gamma)}) by computing 
\begin{equation}
\begin{aligned}
v(c_i^\gamma) \geq \delta(\gamma(v), \hat{\Lambda}_{(i_t)}) + \delta(\hat{\Lambda}_{(i_t)}, \Lambda_{(j)}) \geq \sum_{l = 2}^t \delta(\hat{\Lambda}_{(i_{l-1})}, \hat{\Lambda}_{(i_l)}) + &\delta(\hat{\Lambda}_{(i_t)}, \hat{\Lambda}_{(j)}) \\
&\geq t(\gamma) \min\{\delta(\hat{\Lambda}_{(l)}, \hat{\Lambda}_{(m)})\}_{l \neq m}.
\end{aligned}
\end{equation}
\end{proof}

\subsection{Theta functions} \label{sec proof of main theta}

The $K$-analytic isomorphism between the quotient of the set of non-limit points $\Omega$ modulo the action of a $p$-Whittaker group $\Gamma_0$ and the projective line $\proj_K^1$ is made explicit by a special type of infinite product function.  We define (as in \cite[\S II.2]{gerritzen2006schottky}, with slightly different notation) the \emph{theta function} $\Theta_{a, b}^G$ with respect to any subgroup $G < \PGL_2(K)$ and any choice of elements $a, b$ of its subset $\Omega_G \subset \proj_K^1$ of non-limit points as 
\begin{equation*}
\Theta_{a, b}^G(z) = \prod_{\gamma \in G} \frac{z - \gamma(a)}{z - \gamma(b)}.
\end{equation*}
Here and below, we adopt the convention that if exactly one of the terms in the numerator (resp. denominator) is $\infty$, then the numerator (resp. denominator) is replaced by $1$ and that if the denominator comes out to $0$, then the infinite product equals $\infty \in \proj_K^1$.  We will always assume for our purposes that we have $b \notin G(a)$ and $\infty \notin G(a) \cup G(b)$ (\textit{i.e.} that $a$ and $b$ are in separate orbits under the action of $G$ and that neither is in the orbit of $\infty$).  It is immediate to see that the set of zeros (resp. poles) of $\Theta^G_{a, b}$ coincides with $\{\gamma(a_i)\}_{\gamma \in G}$ (resp. $\{\gamma(b_i)\}_{\gamma \in G}$); in particular, these theta functions are not constant.

In our situation, we are concerned with the theta functions $\Theta_{a, b}^\Gamma$ and $\Theta_{a, b}^{\Gamma_0}$ corresponding to our Schottky and $p$-Whittaker groups $\Gamma \lhd \Gamma_0 < \PGL_2(K)$, where $a, b \in \Omega := \Omega_\Gamma = \Omega_{\Gamma_0}$ satisfy $b \notin \Gamma_0(a)$ and $\infty \notin \Gamma_0(a) \cup \Gamma_0(b)$.  It is shown in \cite[\S II.2, IX.2]{gerritzen2006schottky} that for such $a, b$, the functions $\Theta_{a, b}^\Gamma$ and $\Theta_{a, b}^{\Gamma_0}$ are meromorphic on $\Omega$; moreover, a primarily group-theoretic argument in \cite[\S VIII.1]{gerritzen2006schottky} demonstrates that $\Theta_{a, b}^{\Gamma_0}$ is invariant under the action of $\Gamma_0$, \textit{i.e.} we have $\Theta_{a, b}^{\Gamma_0}(\gamma(z)) = \Theta_{a, b}^{\Gamma_0}(z)$ for $\gamma \in \Gamma_0$.  These above results imply that $\Theta_{a, b}^{\Gamma_0} : \Omega \to \proj_K^1$ induces a map $\vartheta_{a, b} : \Omega / \Gamma_0 \to \proj_K^1$; as the only poles of $\Theta_{a, b}^{\Gamma_0}$ are simple poles occuring at the elements of $\Gamma_0(b)$, the induced function $\vartheta_{a, b}$ has exactly $1$ simple pole and so it is an isomorphism.

For our purposes, the elements $a, b \in \Omega$ will be chosen to be $a_i, b_i \in S$ for some index $i \in \{0, \dots, g\}$, where $S$ is an optimal set whose associated Schottky and $p$-Whittaker groups are $\Gamma \lhd \Gamma_0$.  This is a valid choice of $a, b$ for the following reasons.  We know that $S \subset \Omega$ from \Cref{cor S in Omega}.  We know moreover that $s_i(a_i) = a_i \neq b_i$ and that $\eta_{\gamma(a_i)} = \gamma(\eta_{a_i}) = \notin \Sigma_S \ni \eta_{b_i}$ for any $\gamma \in \Gamma \smallsetminus \{1\}$ thanks to \Cref{lemma off the convex hull variant}.  Since the group $\Gamma_0$ is generated by its subgroup $\Gamma$ and the element $s_i$, we get $b_i \notin \Gamma_0(a_i)$.

From the decomposition $\Gamma_0 = \bigsqcup_{0 \leq n \leq p - 1} \Gamma s_i^n$ and the fact that the element $s_i \in \Gamma_0$ fixes $a_i, b_i$, one sees immediately from formulas that we have 
\begin{equation} \label{eq Theta^p}
\Theta_{a_i, b_i}^{\Gamma_0} = (\Theta_{a_i, b_i}^\Gamma)^p.
\end{equation}

To state and demonstrate the results in the rest of this section, we introduce the following notation.  Given a point $a \in \cc_K$ and a real number $r$, we denote the disc $\{z \in \cc_K \ | \ v(z - a) \geq r\}$ by $D_a(r)$ and the corresponding point of $\Berk$ by $\eta_a(r)$.  Given another real number $s > r$, we introduce the notation $A_a(r, s) \subset \cc_K$ for the \emph{open annulus} given by $\{z \in \cc_K \ | \ r < v(z - a) < s\}$; we will refer to $a$ as a \emph{center} of the annulus $A_a(r, s)$ even though $a \notin A_a(r, s)$.

There is a natural way to extend any rational function $\Theta$ on $\proj_{\cc_K}^1$ (viewed as the subspace of $\Berk$ consisting of the points of Type I) to a function $\Theta_*$ on $\Berk$ using the original seminorm definition of the points of $\Berk$: it is done by composing $\Theta$ with each seminorm; see \cite[Definition 7.2]{benedetto2019dynamics}.  It is clear from the construction and the proofs of the results given in \cite[\S7.1,7.2]{benedetto2019dynamics} that this can be generalized to the case where $\Theta$ is not necessarily a rational function but is meromorphic on a subspace $Z \subset \proj_{\cc_K}^1$ which satisfies that, for any center $a \in Z$ and any real number $r$, there is some $\varepsilon > 0$ such that the function $\Theta$ is meromorphic on each of the annuli $A_a(r, r + \varepsilon)$ and $A_a(r - \varepsilon, r)$; in this generalization, one expects the extended function $\Theta_*$ to be defined on the convex hull of $Z$ in $\Berk$.  We claim that this is the case when $\Theta$ is one of the theta functions $\Theta_{a, b}^\Gamma$ described in \S\ref{sec proof of main theta} above, given a $p$-superelliptic set $S$ with associated $p$-Whittaker group $\Gamma_0$ and elements $a, b \in \Omega = \Omega_\Gamma = \Omega_{\Gamma_0}$, and that the values of the induced function $\Theta_* = (\Theta_{a, b}^\Gamma)_*$ can be computed at any point $\eta_D$ in the convex hull of $\Omega$ in $\Berk$ via \cite[Theorem 7.12, Remark 7.14]{benedetto2019dynamics}.  This result requires in particular that $\Theta$ be meromorphic when restricted to a small enough annulus with inner or outer radius equal to the radius of $D$ and which is centered at a center of $D$, and it states that the image $\Theta(A)$ is also an annulus.

As we do not need to show that the function $\Theta_*$ defined on a certain subspace of $\Berk$ via the construction given in \cite[Theorem 7.12]{benedetto2019dynamics} is actually the map on seminorms induced by $\Theta$ in the sense of \cite[Definition 7.2]{benedetto2019dynamics} or that it is defined on the convex hull of $\Omega$, we leave out the details of such arguments.  However, in order to define $\Theta_*$ using \cite[Theorem 7.12, Remark 7.14]{benedetto2019dynamics} on the subspace $\Sigma_S \subset \Berk$, it is necessary and sufficient to establish the below facts.  (This will also show that the theta function $\tilde{\Theta} := \Theta_{a, b}^{\Gamma_0}$, being the composition of a polynomial function with $\Theta$ as in (\ref{eq Theta^p}), induces a map $\tilde{\Theta}_*$ on $\Sigma_S$, as this construction is functorial with respect to composition.)

\begin{prop} \label{prop meromorphic}

Assume the above set-up and notation, that the set $S$ is optimal, and that $\infty \in S$; choose $a, b \in \Omega$.  For each point $\eta_D \in \Sigma_S \smallsetminus \{\eta_{a_i}, \eta_{b_i}\}_{0 \leq i \leq g}$, the function $\Theta_{a, b}^\Gamma$ is meromorphic on an annulus of the form $A_c(d(D), d(D) + \varepsilon)$ or $A_c(d(D) - \varepsilon, d(D))$ for some $c \in D$ and $\varepsilon > 0$.

\end{prop}

As we have discussed in \S\ref{sec proof of main theta}, a theta function $\Theta_{a, b}^\Gamma$ is meromorphic on $\Omega \subset \proj_K^1$.  Therefore, \Cref{prop meromorphic} is proved immediately from the following proposition and corollary.

\begin{prop} \label{prop element at prescribed distance in Omega}

Let $S \subset \proj_K^1$ be an optimal subset with $\infty \in S$.  Choose an element $a_i \in S$ and a real number $r$ such that the point $\eta_{a_i}(r) \in \Berk$ lies in $\Sigma_S$ but is not a vertex.  Let $a \in \cc_K$ be an element satisfying $v(a - a_i) = r$.  Then the closest point in $\Sigma_S$ to $\eta_a$ is $\eta_{a_i}(r)$, and we have $a \in \Omega$.

\end{prop}

\begin{proof}

Let $\xi$ be the closest point in $\Sigma_S$ to $\eta_a$.  Then we have $\xi \in [\eta_a, \eta_a(r)]$, so that the disc $D \subset \cc_K$ corresponding to $\xi$ satisfies $a \in D \subseteq D_{a_i}(r)$.  Suppose that $\xi \neq \eta_{a_i}(r)$, so that we have $D \subsetneq D_{a_i}(r)$.  Then the logarithmic radius of $D$ is $< r$ so that we have $a_i \notin D$ and that the smallest disc containing $D$ and $a_i$ is $D_{a_i}(r)$.  The half-open segment $[\xi, \eta_{a_i}(r)] \smallsetminus \{\eta_{a_i}(r)\} \subset \Sigma_S$ clearly intersects a neighborhood of the point $\eta_{a_i}(r)$.  But since $\eta_{a_i}(r)$ is not a natural vertex, every sufficiently small neighborhood of $\eta_{a_i}(r)$ in the space $\Sigma_S$ is contained in a path $[\eta_{a_i}(r + \varepsilon), \eta_{a_i}(r - \varepsilon)] \subsetneq [\eta_{a_i}, \eta_\infty] \subset \Sigma_S$ for some small $\varepsilon > 0$.  This is a contradiction, so the closest point in $\Sigma_S$ to $\eta_a$ must be $\eta_{a_i}(r)$.  Then \Cref{cor S in Omega} says that we have $a \in \Omega$.
\end{proof}

\begin{cor} \label{cor annuli in Omega}

Let $S \subset \proj_K^1$ be an optimal subset with $\infty \in S$.  Choose an element $a_i \in S$ and a real number $r$ such that the point $\eta_{a_i}(r) \in \Berk$ lies in $\Sigma_S$.  Then there exist $\varepsilon > 0$ such that we have 
\begin{equation}
A_{a_i}(r, r + \varepsilon), A_{a_i}(r - \varepsilon, r) \subset \Omega.
\end{equation}

\end{cor}

\begin{proof}

Since there are only finitely many vertices of $\Sigma_S$, for small enough $\varepsilon > 0$, there is no vertex of $\Sigma_S$ in the punctured neighborhood $B(\{\eta\}, \varepsilon) \smallsetminus \{\eta\}$.  Choosing any $r' \in \rr$ with $r < r' \leq r + \varepsilon$, we have that the point $\eta_{a_i}(r') \in [\eta_a, \eta_\infty] \subset \Sigma_S$ is not a vertex.  Then \Cref{prop element at prescribed distance in Omega} says that for any $a \in \proj_K^1$ with $v(a - a_i) = r'$, we have $a \in \Omega$.  The claim that $A_{a_i}(r, r + \varepsilon) \subset \Omega$ follows, and the claim that $A_{a_i}(r - \varepsilon, r)$ follows from a similar argument.
\end{proof}

\subsection{An approximation for some values of the theta function} \label{sec proof of main approximation}

Having set up all of the background results that we will need, we now set out to prove \Cref{thm main berk}.  The goal of this section is to prove \Cref{thm one segment} below, which will be crucial to our computations.  In equations below, the expression [h.v.t.] (``higher-valuation terms'') will often appear in sums of elements of $\cc_K$; this means means we are adding an unspecified term whose valuation is higher than that of each of the other terms in the sum.  The following lemma will be used in the proof of \Cref{thm one segment} and also serves to set up some notation used in the theorem.

\begin{lemma} \label{lemma at most one n}

With all of the above notation, suppose that $S \subset \proj_K^1$ is an optimal subset satisfying that $a_j = 0$ and $b_j = \infty$ for some index $j$.  Choose an element $a \in \proj_{\cc_K}^1$, and for $0 \leq n \leq p - 1$, let $\eta_{(n)}^a$ denote the closest point in the convex hull $\Sigma_S$ to $\eta_{\zeta_p^n a}$.  Then all points $\eta_{(n)}^a$ share the same closest point in $\hat{\Lambda}_{(j)}$, and there is at most one $n$ such that $\eta_{(n)}^a \notin \hat{\Lambda}_{(j)}$.

\end{lemma}

\begin{proof}

As we have $\hat{\Lambda}_{(j)} \subset \Sigma_S$, for any $n \in \{0, \dots, p - 1\}$, the closest point in $\hat{\Lambda}_{(j)}$ to $\eta_{(n)}^a$ is clearly the closest point in $\hat{\Lambda}_{(j)}$ to $\eta_{\zeta_p^n a}$.  Meanwhile, we may take $s_0 \in \PGL_2(K)$ to be the automorphism $z \mapsto \zeta_p z$, so for any two indices $n, n'$, we have $\eta_{\zeta_p^{n'} a} = s_0^{n' - n}(\eta_{\zeta_p^n a})$.  Now by applying \cite[Proposition 2.6(d)]{yelton2024branch}, we get that $\eta_{\zeta_p^n a}$ and $s_0^{n' - n}(\eta_{\zeta_p^{n} a} = \eta_{\zeta_p^{n'} a}$ share the same closest point in $\hat{\Lambda}_{(j)}$ (which we now denote by $\xi^a$), and the first statement of the lemma follows.

Let $m$ be an index such that $\delta(\eta_{(m)}^a, \Lambda_{(j)}) \geq \delta(\eta_{(n)}^a, \Lambda_{(j)})$ for $0 \leq n \leq p - 1$.  If $\eta_{(m)}^a \in \hat{\Lambda}_{(j)}$, then we have $\eta_{(n)}^a \in \hat{\Lambda}_{(j)}$ for all other indices $n$ as well and we are done, so assume instead that we have $\eta_{(m)}^a \notin \hat{\Lambda}_{(j)}$.  Choose any other index $n \neq m$.  Now, keeping in mind that $\eta_{\zeta_p^m a} = s_0^{m - n}(\eta_{\zeta_p^n a})$, by \cite[Proposition 2.6(d)]{yelton2024branch} we have 
\begin{equation}
[\eta_{(n)}^a, \xi^a] \cup [\xi^a, s_0^{m - n}(\eta_{(n)}^a)] = [\eta_{(n)}^a, s_0^{m - n}(\eta_{(n)}^a)] \subset [\eta_{\zeta_p^n a}, \eta_{\zeta_p^m a}] = [\eta_{\zeta_p^n a}, \xi^a] \cup [\xi^a, \eta_{\zeta_p^m a}],
\end{equation}
and therefore $s_0^{m - n}(\eta_{(n)}^a) \in [\xi^a, \eta_{\zeta_p^m a}]$.  Since we have 
\begin{equation}
\delta(s_0^{m - n}(\eta_{(n)}^a), \xi^a) = \delta(s_0^{m - n}(\eta_{(n)}^a), s_0^{m - n}(\xi^a)) = \delta(\eta_{(n)}^a, \xi^a) \leq \delta(\eta_{(m)}^a, \xi^a),
\end{equation}
this gives us $s_0^{m - n}(\eta_{(n)}^a) \in [\eta_{(m)}^a, \xi^a] \subset \Sigma_S$.  Therefore, we have the inclusion $[\eta_{(n)}^a, s_0^{m - n}(\eta_{(n)}^a] = [\eta_{(n)}^a, \xi^a] \cup [\xi^a, s_0^{m - n}(\eta_{(n)}^a] \subset \Sigma_S$, and now if $\eta_{(n)}^a \notin \hat{\Lambda}_{(j)}$ (so that $s_0^{m - n}(\zeta_{(n)}) \notin \hat{\Lambda}_{(j)}$ as well by \cite[Proposition 2.6(c)]{yelton2024branch}), one sees using \cite[Proposition 3.11, Definition 3.12]{yelton2024branch} that this contradicts the fact that $S$ is optimal.
\end{proof}

\begin{thm} \label{thm one segment}

Suppose that $S \subset \proj_K^1$ is an optimal subset satisfying the following conditions:
\begin{enumerate}[(i)]
\item we have $0 =: a_j, \infty =: b_j, 1 =: b_i \in S$ for some indices $i \neq j$; 
\item the point $\mathfrak{v}_j \in \Lambda_{(j)}$ corresponding to the disc $\{z \in \cc_K \ | \ v(z) \geq 0\}$ consisting of the integral elements is a distinguished vertex of $\Sigma_{S, 0}$; and 
\item the point $\mathfrak{v}_i := \eta_1 \vee \eta_{a_i} \in \Lambda_{(i)}$ (which corresponds to the disc $\{z \in \cc_K \ | \ v(z - 1) \geq v(a_i - 1)\}$) is a distinguished vertex of $\Sigma_{S, 0}$ satisfying $[\mathfrak{v}_i, \mathfrak{v}_j] \subseteq \Sigma_{S, 0} \smallsetminus \hat{\Lambda}_{(l)}$ for each $l \neq i, j$.
\end{enumerate}

Write $\lambda = a_i - b_i = a_i - 1$.  Given $a \in \Omega$, we define the points $\eta_{(n)}^a$ as in \Cref{lemma at most one n}.  We may approximate $\Theta(a)$ for certain inputs $a$ as follows.

\begin{enumerate}[(a)]

\item For small enough $\nu > 0$, and for any $a \in \Omega$ such that we have $\eta_{(n)}^a \in \Lambda_{(i)} \cup [\mathfrak{v}_i, \mathfrak{v}_j] \cup B(\{\mathfrak{v}_j\}, \nu)$ for some $n \in \{0, \dots, p - 1\}$, we may make the approximation 
\begin{equation} \label{eq general approximation}
\Theta(a) = 1 + p \lambda (1 - a^p)^{-1} + \hvt.
\end{equation}
Moreover, the higher-valuation terms appearing in (\ref{eq general approximation}) above have valuation $> v(p) + v(\lambda) - v(1 - a^p) + \nu$.

\item Given $a \in \Omega$ and $n \in \{0, \dots, p - 1\}$ as in part (a), assume that we have $\eta_{(n)}^a \in (\Lambda_{(i)} \cup [\mathfrak{v}_i, \mathfrak{v}_j]) \smallsetminus \hat{\Lambda}_{(j)}$.  Then we may make the approximation 
\begin{equation} \label{eq better approximation}
\Theta(a) = 1 + \lambda (1 - a)^{-1} + \hvt.
\end{equation}

\end{enumerate}

\end{thm}

\begin{proof}

Define the subset $G_j \subset \Gamma$ as in the proof of \Cref{lemma N_gamma}, where it was observed that $\Gamma$ can be written as the disjoint union $\bigsqcup_{n = 0}^{p - 1} s_j^n G_j s_i^{-n}$.  Now our formula for the theta function $\Theta_{a_i, b_i}^\Gamma$ can be written as 
\begin{equation} \label{eq theta formula broken up}
\Theta_{a_i, b_i}^\Gamma(a) = \prod_{n = 0}^{p - 1} \Big(\prod_{\gamma \in s_j^n G_j s_i^{-n}} \frac{a - \gamma(1 + \lambda)}{a - \gamma(1)}\Big). 
\end{equation}
Now note that we may choose $s_j \in \Gamma_0$ to be the automorphism given by $z \mapsto \zeta_p z$ as it is of order $p$ and fixes $0, \infty \in \proj_K^1$.  Since the automorphism $s_i$ meanwhile fixes the points $a_i = 1 + \lambda, b_i = 1$, we may write the expression for $\Theta(a)$ in (\ref{eq theta formula broken up}) as 
\begin{equation} \label{eq theta of 0 eta distinguished}
\Theta_{a_i, b_i}^\Gamma(a) = \prod_{n = 0}^{p - 1} \Big(\prod_{\gamma \in G_j} \frac{a - \zeta_p^n \gamma(1 + \lambda)}{a - \zeta_p^n \gamma(1)}\Big) = \prod_{\gamma \in G_j} \Big(\prod_{n = 0}^{p - 1} \frac{a - \zeta_p^n \gamma(1 + \lambda)}{a - \zeta_p^n \gamma(1)}\Big) = \prod_{\gamma \in G_j} \frac{a^p - [\gamma(1 + \lambda)]^p}{a^p - [\gamma(1)]^p}.
\end{equation}

Fix any element $\gamma \in G_j$, and define $c_i^\gamma$ as in \Cref{lemma N_gamma}.  \Cref{lemma N_gamma}(b) says that we have $v(c_i^\gamma) \geq v(\lambda)$; in turn, by applying \Cref{prop valuations}(c) and using the property of being clustered in $\frac{v(p)}{p - 1}$-separated pairs, we obtain $v(\lambda) = \delta(\mathfrak{v}_i, \Lambda_{(j)}) \geq \delta(\Lambda_{(i)}, \Lambda_{(j)}) > \frac{2v(p)}{p - 1}$.  

So, taking into account the inequality $v(c_i^\gamma) > \frac{v(p)}{p - 1}$, we may compute the approximation 
\begin{equation} \label{eq approximation}
\begin{aligned}
\frac{a^p - [\gamma(1 + \lambda)]^p}{a^p - [\gamma(1)]^p} &= \frac{a^p - [\gamma(1)]^p(1 + c_i^\gamma)^p}{a^p - [\gamma(1)]^p} = \frac{a^p - [\gamma(1)]^p - [\gamma(1)]^p(pc_i^\gamma + \dots + (c_i^\gamma)^p)}{a^p - [\gamma(1)]^p} \\
 = 1 + &[\gamma(1)]^p([\gamma(1)]^p - a^p)^{-1}(pc_i^\gamma + \dots + (c_i^\gamma)^p) \\
  = 1 + &pc_i^\gamma[\gamma(1)]^p([\gamma(1)]^p - a^p)^{-1} + \hvt.
\end{aligned}
\end{equation}

We now set out to show that, under the hypothesis of part (a), the approximation in (\ref{eq approximation}) gives an element of $\cc_K$ which is farthest from $1$ only for $\gamma = 1$, implying that the term in the product formula for $\Theta_{a_i, b_i}^\Gamma(a)$ in (\ref{eq theta of 0 eta distinguished}) corresponding to $\gamma = 1$ is the one which dominates; the approximation in (\ref{eq general approximation}) then follows directly from applying (\ref{eq approximation}) to $\gamma = 1$.  For the second statement of part (a), it is necessary and sufficient to show something more: that the difference between $1$ and the term in the aforementioned product formula for $\gamma \in G_j \smallsetminus \{1\}$ has valuation exceeding that of the analogous difference for $\gamma = 1$ by more than $\nu$ if $\nu$ is chosen small enough.  Equivalently, we will show under the hypotheses of (a) and (b) that the inequality $v(pc_i^\gamma[\gamma(1)]^p([\gamma(1)]^p - a^p)^{-1}) > v(p\lambda(1 - a^p)^{-1}) + \nu$ holds for all $\gamma \in G_j \smallsetminus \{1\}$ for small enough $\nu$.  With a few straightforward algebraic computations and rearrangements, for each $\gamma \in G_j \smallsetminus \{1\}$, this inequality can be rewritten as 
\begin{equation} \label{eq desired inequality}
v(c_i^\gamma) - v(\lambda) > \sum_{n = 0}^{p - 1} [v(\gamma(1) - \zeta_p^n a) - v(\gamma(1)) - v(1 - \zeta_p^n a)] + \nu.
\end{equation}

It follows from \Cref{lemma N_gamma}(a)(b) that $\{v(c_i^\gamma) - v(\lambda)\}_{\gamma \in \Gamma \smallsetminus \{1\}} \subset \qq_{> 0}$ has a minimum element which is positive.  Choose a positive number 
\begin{equation*}
\nu < \tfrac{1}{p + 1}\min_{\gamma \in \Gamma \smallsetminus \{1\}}\{v(c_i^\gamma) - v(\lambda)\}.
\end{equation*}
After possibly replacing $\nu$ with a smaller positive number, we have $B(\Lambda_{(j)}, \nu) \cap \hat{\Lambda}_{(l)} = \varnothing$ for any $l \neq j$.  After possibly further shrinking $\nu$, we that the closest point in $\Lambda_{(j)}$ to each other space $\hat{\Lambda}_{(l)}$ either has distance $> \nu$ from $\mathfrak{v}_j \in \Lambda_{(j)}$ or is itself the point $\mathfrak{v}_j$.  We will freely assume for the rest of the proof that $\nu$ is small enough for both of these properties to hold.  Let us fix arbitrary choices of $\gamma \in G_j \smallsetminus \{1\}$ and $n \in \{0, \dots, p - 1\}$ and prove the inequality 
\begin{equation} \label{eq desired inequality2}
v(\gamma(1) - \zeta_p^n a) - v(\gamma(1)) - v(1 - \zeta_p^n a) \leq \nu;
\end{equation}
 this clearly implies (\ref{eq desired inequality}) by construction of $\nu$ and the fact that $v(c_i^\gamma) > v(\lambda)$ by \Cref{lemma N_gamma}(b).

As the subspace $\Lambda_{(i)} \cup [\mathfrak{v}_i, \mathfrak{v}_j] \cup B(\{\mathfrak{v}_j\}, \nu) \subset \Berk$ is simply connected and intersects $\hat{\Lambda}_{(j)}$, it follows from \Cref{lemma at most one n} that we have $\eta_{(n)}^a \in \Lambda_{(i)} \cup [\mathfrak{v}_i, \mathfrak{v}_j] \cup B(\{\mathfrak{v}_j\}, \nu)$ for some index $n$ (as is our hypothesis for part (a)) if and only if it holds for all indices $n$.  Assume first that we have $\Lambda_{\gamma(1), \zeta_p^n a} \cap \Lambda_{(j)} \neq \varnothing$.  Then using \Cref{prop valuations}(c), we get 
\begin{equation} \label{eq intersecting axes inequality1}
v(\gamma(1) - \zeta_p^n a) - v(\gamma(1)) \leq v(\gamma(1) - \zeta_p^n a) - \min\{v(\gamma(1)), v(\zeta_p^n a)\} = \delta(\Lambda_{\gamma(1), \zeta_p^n a}, \Lambda_{(j)}) = 0.
\end{equation}
Meanwhile, using \Cref{prop valuations}(a), one sees that our hypotheses imply that $-\nu \leq v(a) \leq \nu$, so we have 
\begin{equation} \label{eq intersecting axes inequality2}
\sum_{n = 0}^{p - 1} v(1 - \zeta_p^n a) \geq -p\nu > v(\lambda) - v(c_i^\gamma) + \nu.
\end{equation}
Putting (\ref{eq intersecting axes inequality1}) and (\ref{eq intersecting axes inequality2}) together gives us the desired inequality (\ref{eq desired inequality2}).

Let us now assume instead that the axes $\Lambda_{\gamma(1), \zeta_p^n a}$ and $\Lambda_{(j)}$ are disjoint.  As in the proof of \Cref{lemma N_gamma}, we have $\gamma(\Lambda_{(i)}) = \Lambda_{\gamma(1), \gamma(1 + \lambda)}$; in particular, we have $\eta_{\gamma(1)} \vee \eta_{\gamma(1 + \lambda)} = \gamma(v)$ for some $v \in \Lambda_{(i)} \subset \Sigma_S$.  Then using \Cref{lemma off the convex hull variant}, we get that the closest point $\xi$ in $\Sigma_S$ to $\eta_{\gamma(1)} \vee \eta_{\gamma(1 + \lambda)}$(and thus to $\Lambda_{\gamma(1), \gamma(1 + \lambda)}$) lies in $\hat{\Lambda}_{(i_t)}$ (with $i_t$ defined as in that lemma) and that we have $\xi \neq \eta_{\gamma(1)} \vee \eta_{\gamma(1 + \lambda)}$ (and thus $\xi \notin \Lambda_{\gamma(1), \gamma(1 + \lambda)}$).  In particular, the point $\xi$ is also the closest point in $\Sigma_S$ to $\gamma(1)$.  From the fact that $i_t \neq j$ and from hypothesis (iii), we have $\xi \notin \Lambda_{(i)} \cup [\mathfrak{v}_i, \mathfrak{v}_j] \cup B(\{\mathfrak{v}_j\}, \nu)$.  Thus, our hypothesis on the point $\eta_{(n)}^a$ shows that $\eta_{(n)}^a \neq \xi$ and that in fact, the point $\xi$ must be farther from $\Lambda_{(j)}$ than $\eta_{(n)}^a$ is.  It follows from all of this that the closest point in $\Lambda_{\gamma(1), \zeta_p^n a}$ to $\Lambda_{(j)}$ is $\eta_{(n)}^a$.  We may now apply \Cref{prop valuations}(a)(c) to get 
\begin{equation} \label{eq zeta a and gamma(1)}
v(\gamma(1) - \zeta_p^n a) - v(\gamma(1)) = \delta(\Lambda_{\gamma(1), \zeta_p^n a}, \Lambda_{(j)}) = \delta(\eta_{(n)}^a, \Lambda_{(j)}).
\end{equation}

From what we have already observed, the points $\eta_{(n)}^a$ and $\gamma(\mathfrak{v}_i)$ share the same closest point $\xi'$ in $\Lambda_{(j)}$; we deduce using \Cref{lemma off the convex hull variant} that $\xi'$ is also the closest point in $\Lambda_{(j)}$ to $\hat{\Lambda}_{(i_t)}$.  Let us assume for the moment that $\eta_{(n)}^a \in B(\Lambda_{(j)}, \nu)$ for $0 \leq n \leq p - 1$.  We then have $\delta(\xi', \mathfrak{v}_j) \leq \delta(\eta_{(n)}^a, \mathfrak{v}_j) \leq \nu$.  Now if $\xi' \neq \mathfrak{v}_j$, we have (by assumption on $\nu$) the contradicting inequality $\delta(\xi', \mathfrak{v}_j) > \nu$.  We therefore have $\xi' = \mathfrak{v}_j$.  Then by \Cref{prop valuations}(a), we have $v(a) = v(\zeta_p^n a) = 0$, and so $v(1 - \zeta_p^n a) \geq 0$.  Meanwhile, we get $v(\gamma(1) - \zeta_p^n a) - v(\gamma(1)) \leq \nu$ from (\ref{eq zeta a and gamma(1)}), and by construction, we have $\nu < \frac{1}{p+1}(v(c_i^\gamma) - v(\lambda))$.  Putting these inequalities together, we see that the desired inequality in (\ref{eq desired inequality2}) holds; we have thus proved part (a) in the case that $\eta_{(n)}^a \in B(\Lambda_{(j)}, \nu)$ for $0 \leq n \leq p - 1$.

Now, as it does not affect the conclusions of the proposition to replace $a$ with $\zeta_p^n a$ for any exponent $n$, let us assume that $\delta(\eta_{(0)}^a, \Lambda_{(j)}) \geq \delta(\eta_{(n)}^a, \Lambda_{(j)})$ for all $n$.  By \Cref{lemma at most one n}, we either have $\eta_{(0)}^a = \dots = \eta_{(p-1)}^a \in \hat{\Lambda}_{(j)}$ or we have $\eta_{(0)}^a \notin \hat{\Lambda}_{(j)}$ and $\eta_{(1)}^a = \dots = \eta_{(p-1)}^a \in [\eta_{(0)}^a, \mathfrak{v}_j] \cap \hat{\Lambda}_{(j)}$.  Now suppose that we have $\eta_{(n)}^a \in [\mathfrak{v}_i, \mathfrak{v}_j]$ for some (and thus for all) $n \in \{0, \dots, p - 1\}$.  Then for each $n$, the closest point in $\Lambda_{1, \zeta_p^n a}$ to $\Lambda_{(j)}$ is $\eta_{(n)}^a$.  We may now apply \Cref{prop valuations}(a)(c) to get 
\begin{equation} \label{eq zeta a and 1}
v(1 - \zeta_p^n a) = v(1 - \zeta_p^n a) - v(1) = \delta(\Lambda_{1, \zeta_p^n a}, \Lambda_{(j)}) = \delta(\eta_{(n)}^a, \Lambda_{(j)}).
\end{equation}
Combining this with (\ref{eq zeta a and gamma(1)}) gives us the equality $v(1 - \zeta_p^n a) = v(\gamma(1) - \zeta_p^n a) - v(\gamma(1))$ for $0 \leq n \leq p - 1$.  This again directly implies the desired inequality in (\ref{eq desired inequality2}) under the assumption that $\eta_{(n)}^a \in [\mathfrak{v}_i, \mathfrak{v}_j]$ for some (all) $n$.  Part (a) is therefore proved.

Now assume the hypothesis of part (b).  Retaining the assumption made above that we have $\delta(\eta_{(0)}^a, \Lambda_{(j)}) \geq \delta(\eta_{(n)}^a, \Lambda_{(j)})$ for all $n$, as before, \Cref{lemma at most one n} tells us that $\eta_{(1)}^a, \dots, \eta_{(p-1)}^a \in \hat{\Lambda}_{(j)}$ and so we must have $\eta_{(0)}^a \in (\Lambda_{(i)} \cup [\mathfrak{v}_i, \mathfrak{v}_j]) \smallsetminus \hat{\Lambda}_{(j)}$.  If $\eta_{(0)}^a \in [\mathfrak{v}_i, \mathfrak{v}_j] \smallsetminus \hat{\Lambda}_{(j)}$, then we clearly have $\eta_{(0)}^a = \eta_1 \vee \eta_a$, whereas if $\eta_{(0)}^a \in \Lambda_{(i)}$, then it is easy to see that $\eta_{(0)}^a \in \Lambda_{(i)} \cap \Lambda_{a, 1} \neq \varnothing$ and $\eta_1 \vee \eta_a \in \Lambda_{(i)}$.  In either case, we get $\eta_1 \vee \eta_a \notin \hat{\Lambda}_{(j)}$.  Applying \Cref{prop valuations}(c), we obtain 
\begin{equation}
v(a - 1) = v(a - 1) - v(1) = \delta(\eta_{(0)}^a, \Lambda_{(j)}) > \frac{v(p)}{p - 1}.
\end{equation}
Now, keeping in mind the above inequality, we compute 
\begin{equation}
p^{-1}(1 - a^p) = p^{-1}(1 - (1 + (a - 1))^p) = p^{-1}(-p(a - 1) + \hvt) = (1 - a) + \hvt.
\end{equation}
Such an approximation is preserved under taking reciprocals, and so the approximation given in part (a) implies the one claimed by part (b).
\end{proof}

\begin{rmk} \label{rmk one segment res char not p}

In the case where the residue characteristic of $K$ is different from $p$, one can show, using a variant of the arguments in the above proof, that the approximation in (\ref{eq general approximation}) holds under the alternate hypothesis that for $1 \leq n \leq p - 1$, we have $[\mathfrak{v}_i, \eta_{(n)}^a] \cap \Lambda_{(l)} = \varnothing$ for all indices $l \neq i, j$.  Indeed, as in the above proof, we may assume that $\eta_{(1)}^a = \dots = \eta_{(p-1)}^a \in \hat{\Lambda}_{(j)} = \Lambda_{(j)}$, which via the equalities in (\ref{eq zeta a and gamma(1)}) and (\ref{eq zeta a and 1}) implies that $\delta(\gamma(1) - \zeta_p^n a) - v(\gamma(1)) = \delta(1 - \zeta_p^n a) = 0$ for $1 \leq n \leq p - 1$.  Now the inequality in (\ref{eq desired inequality}) that is needed for the conclusion of part (a) simplifies to 
\begin{equation} \label{eq desired inequality res char not p}
v(c_i^\gamma) - v(\lambda) > v(\gamma(1) - a) - v(\gamma(1)) - v(1 - a)
\end{equation}
for $\gamma \in G_j \smallsetminus \{1\}$, which can be shown to hold when $\eta_{(0)}^a \notin [\mathfrak{v}_i, \mathfrak{v}_j]$ (the case when $\eta_{(0)}^a \in [\mathfrak{v}_i, \mathfrak{v}_j]$ is already covered by \Cref{thm one segment}(a)) through arguments of a similar flavor, outlined as follows.

As we have $v(1 - a) = 0$, after applying \Cref{prop valuations}(c), the crucial inequality to verify is 
\begin{equation} \label{eq crucial inequality res char not p}
\delta(\Lambda_{\gamma(1), a}, \Lambda_{(j)}) < v(c_i^\gamma) - v(\lambda).
\end{equation}
If the path $[\eta_{\gamma(1)}, \eta_{(0)}^a]$ intersects $\Lambda_{(j)}$, then we have $\delta(\Lambda_{\gamma(1), a}, \Lambda_{(j)}) = 0$, and (\ref{eq crucial inequality res char not p}) holds as $v(c_i^\gamma) - v(\lambda)$ is positive thanks to \Cref{lemma N_gamma}(b).  If, on the other hand, we have $[\eta_{\gamma(1)}, \eta_{(0)}^a] \cap \Lambda_{(j)} = \varnothing$, one verifies using \Cref{lemma off the convex hull variant} that the path $[\mathfrak{v}_i, \Lambda_{(i_t)}]$ intersects $\Lambda_{(j)}$, so that $\delta(\mathfrak{v}_i, \Lambda_{(i_t)}) > \delta(\mathfrak{v}_i, \Lambda_{(j)}) = v(\lambda)$.  Then, using (\ref{eq estimating v(N_gamma)}), we get $\delta(\Lambda_{(i_t)}, \Lambda_{(j)}) < v(c_i^\gamma) - v(\lambda)$.  Meanwhile, the fact that $\eta_{(0)}^a \in [\eta_{\gamma(1)}, \Lambda_{(j)}]$ implies $\delta(\Lambda_{\gamma(1), a}, \Lambda_{(j)}) \leq \delta(\Lambda_{(i_t)}, \Lambda_{(j)})$, so we get the desired inequality (\ref{eq crucial inequality res char not p}).

\end{rmk}

\subsection{The image of one segment of the convex hull under the theta function} \label{sec proof of main one segment}

Our next step is to use \Cref{thm one segment} to compute the values on a particular subspace of $\Sigma_S$ of the function $(\Theta_{a_i, b_i}^\Gamma)_*$ (the existence of which was established in \S\ref{sec proof of main theta}) induced by the theta function $\Theta_{a_i, b_i}^\Gamma$.

We retain all of the above notation and define, for any subspace $\Lambda \subset \Berk$ and real number $\nu > 0$, the subspace $B(\{\mathfrak{v}_j\}, \nu)^- \subset \Berk$ to be the subspace of the neighborhood $B(\{\mathfrak{v}_j\}, \nu)$ consisting of the points of Type II or III corresponding to discs in $\cc_K$ whose elements all have valuation $\leq 0$.  We remark that, due to \Cref{prop valuations}(a), an alternate (equivalent) definition of the space $B(\{\mathfrak{v}_j\}, \nu)^-$ is that it consists of those points in $B(\{\mathfrak{v}_j\}, \nu)$ whose closest point in $\Lambda_{0, \infty}$ lies in the path $[\mathfrak{v}_j, \eta_\infty]$.

The main result of this subsection is as follows.

\begin{prop} \label{prop one segment}

Suppose that $S \subset \proj_K^1$ is an optimal subset satisfying hypotheses (i)-(iii) of \Cref{thm one segment}.  As in the statement of that theorem, we write $\mathfrak{v}_j = \eta_0(0)$, and we write $\Theta$ for the theta function $\Theta_{a_i, b_i}^\Gamma$.

\begin{enumerate}[(a)]

\item We have $v(\Theta(0) - 1) = v(p) + v(\lambda)$.

\item For sufficiently small $\nu > 0$, the induced map $\Theta_*$ established in \S\ref{sec proof of main theta} on $\Sigma_S$, when restricted to the subspace $[\eta_1, \eta_1(d)] \subset \Sigma_S$, is one-to-one and bicontinuous.  Its values on $[\eta_1, \eta_1(-\nu)]$ are given by 
\begin{equation} \label{eq outputs of eta(1, d) Gamma}
\Theta_*(\eta_1(d)) = 
\begin{cases}
\eta_1(v(p) + v(\lambda) - pd) &\mathrm{for} \ -\nu \leq d \leq \frac{v(p)}{p - 1} \\
\eta_1(v(\lambda) - d) &\mathrm{for} \ d \geq \frac{v(p)}{p - 1}
\end{cases}.
\end{equation}

\item For sufficiently small $\nu > 0$, for each $\eta \in B(\{\mathfrak{v}_j\}, \nu)^- \cap \Sigma_S \smallsetminus ([\eta_1, \mathfrak{v}_j] \cup \Lambda_{(j)})$, the point $\Theta_*(\eta)$ corresponds to a disc which does not contain the element $1$ but is centered at an element $c$ satisfying $v(c - 1) = v(p) + v(\lambda)$.

\end{enumerate}

Consequently, letting 
\begin{equation*}
U^- = [\eta_1, \mathfrak{v}_j] \cup (B(\{\mathfrak{v}_j\}, \nu)^- \cap \Sigma_S),
\end{equation*}
we have 
\begin{equation} \label{eq almost one-to-one}
\Theta_*([\eta_1, \mathfrak{v}_j]) \cap \Theta_*(U^- \smallsetminus [\eta_1, \mathfrak{v}_j]) = \varnothing.
\end{equation}

\end{prop}

In order to prove the above theorem, we first need an elementary result regarding annuli.

\begin{lemma} \label{lemma annuli}

Let $A = A_a(r, s)$ be an (open) annulus, and suppose that there is a point $a' \in \cc_K$ and rational numbers $r' < s'$ such that 
\begin{enumerate}[(i)]
\item for every $z \in A$, we have $r' < v(z - a') < s'$ and 
\item conversely, for every $\rho$ such that $r' < \rho < s'$, there exists $z \in A$ such that $v(z - a') = \rho$.
\end{enumerate}
Then we have $r = r'$ and $s = s'$, and we have $v(a' - a) \geq s'$, so that $A = A_{a'}(r', s')$.

\end{lemma}

\begin{proof}

Given any points $z, w \in A$ such that $v(z - a) \neq v(w - a)$, by the non-archimedean property, we have $v(z - w) = \min\{v(z - a), v(w - a)\}$.  This easily implies 
\begin{equation} \label{eq infemum}
\inf_{z, w \in A} v(z - w) = r.
\end{equation}
By a similar argument, the hypotheses (i) and (ii) of the statement imply that the same infemum in (\ref{eq infemum}) equals $r'$, so we get $r = r'$.  Now the fact that $a' \notin A$ by hypothesis (i) implies that we have either $v(a' - a) \geq s$ or $v(a' - a) \leq r$.  If $v(a' - a) \leq r$, then applying the non-archimedean property and using the fact that $v(z - a') \geq r' = r$ gives us $v(z - a) = \min\{v(z - a'), v(a' - a)\} = v(a' - a) \leq r$, which contradicts the construction of $A$.  We therefore have $v(a' - a) \geq s$.  It follows that $A = A_{a'}(r, s)$; in other words, $a'$ is also a center of the open annulus $A$.  Now it is immediate from hypotheses (i) and (ii) that $s' = s$.
\end{proof}

\begin{proof}[Proof (of \Cref{prop one segment})]

Part (a) immediately follows from the estimation 
\begin{equation}
\Theta(0) = 1 + p \lambda + \hvt
\end{equation}
obtained by from putting $a = 0$ into the approximation given by \Cref{thm one segment}(a).

Choose $\nu > 0$ small enough that the conclusion of \Cref{thm one segment}(a) is satisfied.  Fix a real number $d$ with $-\nu \leq d < \frac{v(p)}{p - 1}$.  For sufficiently small $\varepsilon > 0$, there is no vertex of $\Sigma_S$ in the interior of the path $[\eta_1(d), \eta_1(d + \varepsilon)] \subset \Sigma_S$.  Then for any $\rho$ with $d < \rho < d + \varepsilon$ and any $a \in \cc_K$ with $v(a - 1) = \rho$, by \Cref{prop element at prescribed distance in Omega}, we have $a \in \Omega$.  We observe that 
\begin{equation}
1 - a^p = 1 - ((a - 1) + 1)^p = (1 - a)^p + \hvt.
\end{equation}
By \Cref{prop element at prescribed distance in Omega}, we may apply \Cref{thm one segment}(a) to get 
\begin{equation} \label{eq outputs of eta(1, d) radius of annulus}
v(\Theta(a) - 1) = v(p) + v(\lambda) - p\rho.
\end{equation}
Now \cite[Theorem 7.12]{benedetto2019dynamics}, with the help of \Cref{prop meromorphic}, says that, after possibly shrinking $\varepsilon > 0$, the image $\Theta(A_1(d, d + \varepsilon))$ coincides with an open annulus $A_a(r, s)$ (for some point $a \in \cc_K$ and rational numbers $r < s$).  Meanwhile, one deduces immediately from (\ref{eq outputs of eta(1, d) radius of annulus}) that for every $z \in \Theta(A_1(d, d + \varepsilon))$ we have 
\begin{equation}
v(p) + v(\lambda) - p(d + \varepsilon) < v(z - 1) < v(p) + v(\lambda) - pd.
\end{equation}
Now applying \Cref{lemma annuli}, we get 
\begin{equation}
\Theta(A_1(d, d + \varepsilon)) = A_1(v(p) + v(\lambda) - p(d + \varepsilon), v(p) + v(\lambda) - pd).
\end{equation}
Now the formula claimed in (\ref{eq outputs of eta(1, d) Gamma}) in the case that $0 \leq d < \frac{v(p)}{p - 1}$ follows from applying the last statement of \cite[Theorem 7.12]{benedetto2019dynamics}.

Now fix a real number $d > \frac{v(p)}{p - 1}$.  For sufficiently small $\varepsilon > 0$, there is again no vertex of $\Sigma_S$ in the interior of the path $[\eta_1(d - \varepsilon), \eta_1(d)] \subset \Sigma_S$.  Then for any $\rho$ with $d - \varepsilon < \rho < d$ and any $a \in \cc_K$ with $v(a - 1) = \rho$, again we have $a \in \Omega$ by \Cref{prop element at prescribed distance in Omega}.  It is then straightforward to deduce from the approximation provided by \Cref{thm one segment}(b) that we have 
\begin{equation}
v(\Theta(a) - 1) = v(\lambda) - \rho.
\end{equation}
This time, \cite[Theorem 7.12, Remark 7.14]{benedetto2019dynamics}, with the help of \Cref{prop meromorphic}, says that for small enough $\varepsilon > 0$, the image $\Theta(A_1(d - \varepsilon, d))$ coincides with an open annulus $A_a(r, s)$ (for some point $a \in \cc_K$ and rational numbers $r < s$).  Now the same argument as above involving \Cref{lemma annuli} shows that we get 
\begin{equation}
\Theta(A_1(d - \varepsilon, d)) = A_1(v(\lambda) - d, v(\lambda) - d + \varepsilon).
\end{equation}
Now the formula claimed in (\ref{eq outputs of eta(1, d) Gamma}) in the case that $d \geq \frac{v(p)}{p - 1}$ follows from applying the last statement of \cite[Theorem 7.12]{benedetto2019dynamics} (again with the help of \cite[Remark 7.14]{benedetto2019dynamics}).

We have thus proved the formulas in (\ref{eq outputs of eta(1, d) Gamma}), and it is immediate from these formulas that the claims of being one-to-one and bicontinuous hold, so part (b) is proved.

To prove part (c), again choose $\nu$ to be small enough that the conclusion of \Cref{thm one segment}(a) is satisfied; we may assume that the only distinguished vertex in the neighborhood $B(\{\mathfrak{v}_j\}, \nu)$ is $\mathfrak{v}_j$.  Choose a point $\eta \in B(\{\mathfrak{v}_j\}, \nu)^- \cap \Sigma_S \smallsetminus ([\eta_1, \mathfrak{v}_j] \cup \Lambda_{(j)})$, and let $\xi$ be the closest point in $\Lambda_{(j)}$ to $\eta$.  By definition, the point $\xi$ is a distinguished vertex.  As the path $[\eta, \xi] \cup [\xi, \mathfrak{v}_j] = [\eta, \mathfrak{v}_j]$ is contained in the neighborhood $B(\{\mathfrak{v}_j\}, \nu)$, we must have $\xi = \mathfrak{v}_j$ and $\delta(\eta, \mathfrak{v}_j) \leq \nu$.

From the structure of $\Sigma_S$ and our hypotheses on $\eta$, there is some element $a_l \in S$ for $l \neq i, j$ such that $\eta \in [\mathfrak{v}_j, \eta_{a_l}]$, so that $\eta = \eta_{a_l}(d)$ for some (logarithmic) radius $d \in \rr$.  By \Cref{prop element at prescribed distance in Omega}, there is an element $z_0 \in \Omega$ such that $v(z_0, a_l) = d$ and the closest point in $\Sigma_S$ to $\eta_{z_0}$ is $\eta$.  We apply \Cref{prop valuations}(a) to get $v(z_0) = 0$ and \Cref{prop valuations}(c) to get $v(z_0 - 1) = v(z_0 - 1) - v(1) = 0$ since $\Lambda_{1, z_0} \cap \Lambda_{(j)} = \{\mathfrak{v}_j\} \neq \varnothing$.  Meanwhile, as $\eta = \eta_{z_0}(d) \neq \eta_0(d) \in \Lambda_{(j)}$, we have $\nu \geq d > 0$.

Let us now show that we have $v(z_0^p - 1) = 0$ as well.  If $p$ is the residue characteristic of $K$, then this certainly follows immediately from the fact that $v(z_0 - 1) = 0$, since the reduction of $z_0 - 1$ in the residue field is then a unit, while the reductions of $(z_0 - 1)^p$ and of $z_0^p - 1$ are equal.  We therefore suppose for the moment that $p$ is not the residue characteristic of $K$.  We already have $v(z_0 - 1) = 0$ and $v(z_0 - \zeta_p^n) \geq 0$ for $1 \leq n \leq p - 1$ and only need to show that equality holds to get $v(z_0^p - 1) = \sum_{n = 0}^{p - 1} v(z_0 - \zeta_p^n) = 0$.  \Cref{lemma at most one n} says we have (using the notation of that lemma) $\eta_{(n)}^1 \notin \hat{\Lambda}_{(j)} = \Lambda_{(j)}$ for at most one $n$.  As the path $[\eta_1, \mathfrak{v}_j]$ passes through $\eta_1 \vee \eta_{1+\lambda} \in \Sigma_S$, it is clear that $\eta_{(0)}^1 \notin \Lambda_{(j)}$, so we have $\eta_{(1)}^1 = \dots = \eta_{(p-1)}^1 \in \Lambda_{(j)}$ and, by \Cref{prop valuations}(a), even $\eta_{(1)}^1 = \dots = \eta_{(p-1)}^1 = \mathfrak{v}_j$.  If we have $v(z_0 - \zeta_p^n) =: d' > 0$ for some $n$, then, applying \Cref{prop valuations}(c), we have $\eta_{z_0}(d') = \eta_{z_0} \vee \eta_{\zeta_p^n} \notin \Lambda_{(j)}$.  From $\eta_{z_0}(d) \in \Sigma_S \smallsetminus \Lambda_{(j)}$, we get $\eta_{z_0}(\min\{d, d'\}) \in \Sigma_S \smallsetminus \Lambda_{(j)}$.  But as $\min\{d, d'\} > 0$, this point $\eta_{z_0}(\min\{d, d'\})$ lies in the interior of the path $[\eta_{\zeta_p^n}, \mathfrak{v}_j = \eta_{(n)}^1]$, a contradiction.  Therefore, we have $v(z_0 - \zeta_p^n) = 0$, as desired.

Given any $\rho$ such that $d > \rho > 0$, the point $\eta_{z_0}(\rho)$ lies in the interior of the path $[\eta_{z_0}(d), \mathfrak{v}_j] \subset B(\{\mathfrak{v}_j\}, \nu)$ and so is not a distinguished vertex.  Then by \Cref{prop element at prescribed distance in Omega}, there is an element $a \in \Omega$ such that $v(a - z_0) = \rho$ and the closest point in $\Sigma_S$ to $\eta_a$ is $\eta_{z_0}(\rho)$.  Note that from $v(z_0) = 0$ we get $v(a) = 0$ and so $v(a - \zeta_p^n z_0) \geq 0$ for $0 \leq n \leq p - 1$.  We thus see that $v(a^p - z_0^p) = v(a - z_0) + \sum_{n = 1}^{p - 1} v(a - \zeta_p^n z_0) \geq \rho$.  By the exact same argument as was used to show that $v(z_0^p - 1) = 0$, we have $v(a^p - 1) = 0$.

We may now apply \Cref{thm one segment} to both inputs $z_0$ and $a$ to get 
\begin{equation} \label{eq approximation of Theta(z_0) and Theta(a)}
\Theta(z_0) = 1 + p\lambda(1 - z_0^p)^{-1} + \hvt, \ \ \ \Theta(a) = 1 + p\lambda(1 - a^p)^{-1} + \hvt,
\end{equation}
where in both cases the higher-valuation terms have valuation greater than $v(p) + v(\lambda) - 0 + \nu > v(p) + v(\lambda) + \rho$.  Now we get the estimation 
\begin{equation}
\Theta(a) - \Theta(z_0) = p\lambda[(1 - a^p)^{-1} - (1 - z_0^p)^{-1}] + \hvt,
\end{equation}
where the higher-valuation terms again have valuation greater than $v(p) + v(\lambda) + \rho$.  Using the fact that $v(a^p - z_0^p) > v(1 - z_0^p) = 0$, we compute that the difference $(1 - a^p)^{-1} - (1 - z_0^p)^{-1} = ([1 - z_0^p] - [a^p - z_0^p])^{-1} - (1 - z_0^p)^{-1}$ has valuation $v(a^p - z_0^p) \geq \rho$, so we get  
\begin{equation}
v(\Theta(a) - \Theta(z_0)) \geq v(p) + v(\lambda) + \rho.
\end{equation}

As before, we may apply \cite[Theorem 7.12, Remark 7.14]{benedetto2019dynamics} and \Cref{lemma annuli}; this time we find that 
\begin{equation}
\Theta(A_{z_0}(d - \varepsilon, d)) = A_{\Theta(z_0)}(r, s)
\end{equation}
for $d > \varepsilon > 0$ and with $s > r > v(p) + v(\lambda) + d - \varepsilon > v(p) + v(\lambda)$.  It follows (again using \cite[Theorem 7.12, Remark 7.14]{benedetto2019dynamics}) that the output $\Theta_*(\eta)$ is $\eta_{\Theta(z_0)}(d')$ for some $d' > v(p) + v(\lambda)$.  At the same time, the approximation of $c := \Theta(z_0)$ given in (\ref{eq approximation of Theta(z_0) and Theta(a)}), combined with the fact that $v(1 - z_0^p) = 0$, tells us that $v(\Theta(z_0) - 1) = v(p) + v(\lambda)$.  Therefore the disc corresponding to $\Theta_*(\eta)$ cannot contain $1$.  Thus, we have proved part (c).

Now, by inspecting the formulas in (\ref{eq outputs of eta(1, d)}) given by part (b) for $d \geq 0$, which give output points corresponding to discs centered at $1 \in \cc_K$ and whose logarithmic radii are $\leq v(p) + v(\lambda)$, we conclude from the $-\nu \leq d \leq \frac{v(p)}{p - 1}$ case of part (b) and from part (c) that we have 
\begin{equation} \label{eq disjointness}
\Theta_*(\eta) \notin \Theta_*([\eta_1, \mathfrak{v}_j]) \text{ for all } \eta \in U^- \smallsetminus [\eta_1, \mathfrak{v}_j].
\end{equation}
We therefore get the final statement of the theorem asserting disjointness of images.
\end{proof}

\begin{cor} \label{cor one segment}

Suppose that $S \subset \proj_K^1$ is an optimal subset satisfying hypotheses (i)-(iii) of \Cref{thm one segment}.  For brevity of notation, write $\tilde{\Theta}$ for the theta function $\Theta_{a_i, b_i}^{\Gamma_0}$.

\begin{enumerate}[(a)]

\item We have $v(\tilde{\Theta}(0) - 1) = 2v(p) + v(\lambda)$.

\item The induced map $\tilde{\Theta}_*$, when restricted to the subspace $[\eta_1, \mathfrak{v}_j] \subset \Berk$, is one-to-one and bicontinuous.  Its values on $[\eta_1, \mathfrak{v}_j]$ are given by 
\begin{equation} \label{eq outputs of eta(1, d)}
\tilde{\Theta}_*(\eta(1, d)) = 
\begin{cases}
\eta(1, 2v(p) + v(\lambda) - pd) &\mathrm{for} \ 0 \leq d \leq \frac{v(p)}{p - 1} \\
\eta(1, v(p) + v(\lambda) - d) &\mathrm{for} \ \frac{v(p)}{p - 1} \leq d \leq v(\lambda) - \frac{v(p)}{p - 1} \\
\eta(1, pv(\lambda) - pd) &\mathrm{for} \ d \geq v(\lambda) - \frac{v(p)}{p - 1}
\end{cases}.
\end{equation}

\item For sufficiently small $\nu > 0$, defining $U^-$ as in the statement of \Cref{prop one segment}, we have 
\begin{equation} \label{eq almost one-to-one}
\tilde{\Theta}_*([\eta_1, \mathfrak{v}_j]) \cap \tilde{\Theta}_*(U^- \smallsetminus [\eta_1, \mathfrak{v}_j]) = \varnothing.
\end{equation}

\end{enumerate}

\end{cor}

\begin{proof}

Write $\Theta$ for $\Theta^\Gamma_{a_i, b_i}$.  As observed in (\ref{eq Theta^p}) above, we have $\tilde{\Theta} \equiv \Theta^p$; we can therefore express $\tilde{\Theta}$ as the composition $P \circ \Theta$, where $P : \proj_{\cc_K}^1 \to \proj_{\cc_K}^1$ is the $p$-power map $z \mapsto z^p$.  It follows that the induced map $\tilde{\Theta}_*$ is the composition of $\Theta_*$ with the map $P_* : \Berk \to \Berk$ induced by $P$.  By \cite[Proposition 7.6]{benedetto2019dynamics} we have $P_*(\eta_D) = \eta_{P(D)}$ for any disc $D \subset \cc_K$.  There is a well-known formula for the image of any disc under $P$ (see, for instance, \cite[Exercise 7.15]{benedetto2019dynamics}) given by 
\begin{equation} \label{eq formula for P_*}
P(D_a(r)) = 
\begin{cases}
D_{a^p}(pr) \ &\mathrm{for} \ r \leq v(a) + \frac{v(p)}{p - 1} \\
D_{a^p}(v(p) + (p - 1)v(a) + r) \ &\mathrm{for} \ r \geq v(a) + \frac{v(p)}{p - 1}
\end{cases}.
\end{equation}
Now parts (a) and (b) can be confirmed using the fact that $\tilde{\Theta}_* = P_* \circ \Theta_*$ and applying the above formula (\ref{eq formula for P_*}) to the outputs of $\tilde{\Theta}_*$ provided by \Cref{prop one segment}(a)(b).

To prove part (c), we begin by observing directly from \Cref{prop one segment}(a)(b) that we may describe the image of $[\eta_1, \mathfrak{v}_j]$ under $\Theta_*$ as 
\begin{equation}
\Theta_*([\eta_1, \mathfrak{v}_j]) = [\eta(1, v(p) + v(\lambda)), \eta_\infty] = \{\eta(1, d') \ | \ d' \leq v(p) + v(\lambda)\}.
\end{equation}
Choose a point $\eta \in U^- \smallsetminus [\eta_1, \mathfrak{v}_j]$.  If $\eta$ can be written as $\eta(0, d)$ for some (necessarily negative) $d$, then, using \Cref{prop one segment}(b) and the formula in (\ref{eq formula for P_*}), we compute 
\begin{equation}
\begin{aligned}
\tilde{\Theta}_*(\eta(0, d)) = (P_* \circ \Theta_*)(\eta(0, d)) = P_*(\eta(&1, v(p) + v(\lambda) - pd)) = \eta(1, 2v(p) + v(\lambda) - pd) \\
&\notin \{\eta(1, d') \ | \ d' \leq v(p) + v(\lambda)\} = \Theta_*([\eta_1, \mathfrak{v}_j]).
\end{aligned}
\end{equation}
Now suppose on the other hand that $\eta \notin \Lambda_{(j)}$.  Then \Cref{prop one segment}(c) implies that the point $\Theta_*(\eta)$ can be written as $\eta(c, r)$ where $r > v(c - 1) = v(p) + v(\lambda) > \frac{v(p)}{p - 1}$.  By observing that 
\begin{equation}
v(c - \zeta_p^n) = \min\{v(c - 1), v(\zeta_p^n - 1)\} = \min\{r, \tfrac{v(p)}{p - 1}\} = \tfrac{v(p)}{p - 1} < r
\end{equation}
for $1 \leq n \leq p - 1$, we deduce that the point $\Theta_*(\eta)$ corresponds to a disc $D$ not containing $\zeta_p^n$ for any $n$.  Then we have $1 \notin P(D)$, and since $P(D)$ is the disc corresponding to $\tilde{\Theta}_*(\eta) = P_*(\Theta_*(\eta))$, we again get $\tilde{\Theta}_*(\eta) \notin \tilde{\Theta}_*([\eta_1, \mathfrak{v}_j])$.  This completes the proof of part (c).
\end{proof}

\subsection{The image of the whole convex hull under the theta function} \label{sec proof of main whole}

This subsection consists of the rest of the proof of \Cref{thm main berk}.  We will actually prove a slightly more sophisticated statement by more precisely defining the claimed map $\pi_*$ in the following manner.  In the statement of \Cref{thm main berk}, we have implicitly provided a $K$-analytic isomorphism from the quotient of the set of $K$-points $\Omega_{\Gamma_0}(K)$ by the action of $\Gamma_0$ and the superelliptic curve $C / K$; let us extend this to an isomorphism over $\cc_K$ and denote it by $\vartheta : \Omega_{\Gamma_0} / \Gamma_0 \stackrel{\sim}{\to} C / \cc_K$.  Also, in that statement, the map $\pi$ is simply a bijection from the $(2g + 2)$-element set $S$ to the $(2g + 2)$-element set $\mathcal{B}$ of branch points of $C$.  Let us extend $\pi$ to the composition of $\vartheta$ with the quotient map $\Omega_{\Gamma_0} \twoheadrightarrow \Omega_{\Gamma_0} / \Gamma_0$.  If we have $\infty \notin \Omega_{\Gamma_0}$, then we may choose an automorphism $\sigma \in \PGL_2(\cc_K)$ such that $\infty \in \sigma(\Gamma_0)$ and replace $S$ with $\sigma(S)$ (so that the associated objects $\Gamma_0$ and $\Omega_{\Gamma_0}$ are replaced with $\sigma \Gamma_0 \sigma^{-1}$ and $\sigma(\Omega_{\Gamma_0})$ respectively) without affecting the curve $C$ (see \Cref{rmk no encumbrance}) or the structure of the convex hull $\Sigma_S$ (as $\sigma$ acts as a metric-preserving homeomorphism on $\Berk$).  Having done this, we now assume that we have $\infty \in \Omega_{\Gamma_0}$.

\begin{lemma} \label{lemma gluing variant}

With the above set-up, there exists $\tau \in \PGL_2(\cc_K)$ such that the map $\tau \circ \pi : \Omega_{\Gamma_0} \to \proj_K^1$ is the theta function $\Theta^{\Gamma_0}_{a, b}$ for some $a, b \in \Omega_{\Gamma_0}$ with $b \notin \Gamma_0(a)$ and $\infty \notin \Gamma_0(a) \cup \Gamma_0(b)$.

\end{lemma}

\begin{proof}

Let $\tau \in \PGL_2(\cc_K)$ be an automorphism sending $\vartheta \circ \bar{\sigma}(\infty)$ to $1$, and choose elements $a, b \in \Omega_{\Gamma_0}$ whose respective images modulo the action of $\Gamma_0$ are $(\tau \circ \vartheta)^{-1}(0)$ and $(\tau \circ \vartheta)^{-1}(\infty)$.  Then it is clear from formulas for theta functions (more specifically, noting that we the images of $\infty, a, b$ under $\Theta^{\Gamma_0}_{a, b}$ are $1, 0, \infty$ respectively) that the composition of analytic isomorphisms $\tau \circ \vartheta \circ (\vartheta^{\Gamma_0}_{a, b})^{-1}$ fixes each of the points $1, 0, \infty \in \proj_{\cc_K}^1$.  The only analytic automorphism of the projective line fixing $3$ distinct points is the identity, so we get $\tau \circ \vartheta = \vartheta^{\Gamma_0}_{a, b}$.
\end{proof}

The above lemma says that the function $\pi$ is the composition $\tau^{-1} \circ \Theta^{\Gamma_0}_{a, b}$ for some $\tau \in \PGL_2(K)$ and $a, b \in \Omega_{\Gamma_0}$.  We have seen in \S\ref{sec proof of main theta} that the map $\Theta^{\Gamma_0}_{a, b}$ induces a map $(\Theta^{\Gamma_0}_{a, b})_* : \Sigma_S \to \Berk$; now by functoriality of this construction with respect to compositions of functions, the induced function $\pi_* = \tau^{-1} \circ (\Theta^{\Gamma_0}_{a, b})_*$ is defined on $\Sigma_S$ (here $\tau = \tau_*$ is the usual extension of $\tau$ to a metric-preserving self-homeomorphism of $\Berk$).  As $\tau^{-1}$ respects distances, is now clear that it suffices to prove the assertions of \Cref{thm main berk} under the assumption that $\tau = 1$, or in other words, that $\pi_* = (\Theta^{\Gamma_0}_{a, b})_*$.

From now on, we abbreviate $\Theta^{\Gamma_0}_{a, b}$ as $\tilde{\Theta}$ and set out to prove that the assertions of \Cref{thm main berk} hold for the map $\tilde{\Theta}_* : \Sigma_S \to \Berk$.  We also write $\mathfrak{v} \in \Berk$ for the point $\eta_1(0) = \eta_0(0)$.

Our strategy is now to deduce \Cref{thm main berk} from \Cref{cor one segment} (which asserts that the desired conclusions hold when restricted to a certain subspace $U^-$ of the convex $\Sigma_S$ where the set $S$ satisfies certain additional hypotheses) using a ``gluing method'' which exploits the fact that (loosely speaking) images of translations of this subspace $U^-$ cover the whole convex hull $\Sigma_S$.  The following lemma is crucial to the gluing process.

\begin{lemma} \label{lemma gluing}

Let $S \subset \proj_K^1$ be a $p$-superelliptic set with associated $p$-Whittaker group $\Gamma_0 < \PGL_2(K)$.  Let $\sigma \in \PGL_2(K)$ be a fractional linear transformation, and choose elements $a, b, a', b' \in \Omega_{\Gamma_0}$ with $b \notin \Gamma_0(a)$ and $b' \notin \Gamma_0(a')$ and $\infty \notin \Gamma_0(a) \cup \Gamma_0(b) \cup \Gamma_0(a') \cup \Gamma_0(b')$.  The automorphism $\sigma$ maps the set of non-limit points of $\Gamma_0$ to those of its conjugate $\Gamma_0^\sigma := \sigma \Gamma_0 \sigma^{-1}$, and there is a fractional linear transformation $\tau \in \PGL_2(\cc_K)$ and an analytic isomorphism $\bar{\sigma}$ such that the below diagram commutes, where $\pi_{\Gamma_0}$ and $\pi_{\Gamma_0^\sigma}$ are the obvious quotient maps.
\begin{equation} \label{eq gluing commutative diagram}
\xymatrixcolsep{5pc}\xymatrix{ \Omega_{\Gamma_0} \ar[r]^\sim_\sigma \ar@{->>}[d]_{\pi_{\Gamma_0}} \ar@/_2.5pc/@{->>}[dd]_{\tilde{\Theta}^{\Gamma_0}_{a, b}} &\Omega_{\Gamma_0^\sigma} \ar@{->>}[d]_{\pi_{\Gamma_0^\sigma}} \ar@/^2.5pc/@{->>}[dd]^{\tilde{\Theta}^{\Gamma_0^\sigma}_{\sigma(a'), \sigma(b')}} \\
\Omega_{\Gamma_0} / \Gamma_0 \ar[r]^\sim_{\bar{\sigma}} \ar[d]^\wr_{\vartheta^{\Gamma_0}_{a, b}} &\Omega_{\Gamma_0^\sigma} / \Gamma_0^\sigma \ar[d]^\wr_{\vartheta^{\Gamma_0^\sigma}_{\sigma(a'), \sigma(b')}} \\
\proj_{\cc_K}^1 \ar[r]^\sim_\tau &\proj_{\cc_K}^1 }
\end{equation}

\end{lemma}

\begin{proof}

It is an easy exercise to directly check that we have $\sigma(\Omega_{\Gamma_0}) = \Omega_{\Gamma_0^\sigma}$ (and that we have $\sigma(b') \notin \Gamma_0^\sigma(\sigma(a'))$); the automorphism $\sigma: \proj_{\cc_K}^1 \to \proj_{\cc_K}^1$ thus restricts to an analytic isomorphism $\sigma: \Omega_{\Gamma_0} \to \Omega_{\Gamma_0^\sigma}$.  This clearly induces a well-defined map $\bar{\sigma} : \Omega_{\Gamma_0} / \Gamma_0 \to \Omega_{\Gamma_0^\sigma} / \Gamma_0^\sigma$ which is also an analytic isomorphism.  Now the composition $\tau := \vartheta^{\Gamma_0^\sigma}_{\sigma(a'), \sigma(b')} \circ \bar{\sigma} \circ (\vartheta^{\Gamma_0}_{a, b})^{-1}: \proj_{\cc_K}^1 \to \proj_{\cc_K}^1$ of analytic isomorphisms is an analytic automorphism of $\proj_{\cc_K}^1$ and is therefore a fractional linear transformation.
\end{proof}


For the rest of this subsection, we call an ordered pair $(v, w)$ of distinguished vertices a \emph{neighboring pair} (of distinguished vertices of $\Sigma_{S, 0}$) if $v$ and $w$ do not lie in the same axis $\Lambda_{(i)}$ and if the path $[v, w] \subset \Sigma_S$ contains no other distinguished vertex.  Given any neighboring pair $(v, w)$ and any real number $\nu > 0$, define the subspace $B(\{w\}, \nu)^+ \subset B(\{w\}, \nu)$ (resp. $B(\{w\}, \nu)^- \subset B(\{w\}, \nu)$) to consist of the points whose closest point in $\Lambda_{(j)}$ lies in the half-axis $[w, \eta_{a_j}]$ (resp. $[w, \eta_{b_j}]$); note that we have $B(\{w\}, \nu)^+ \cup B(\{w\}, \nu)^- = B(\{w\}, \nu)$.  Now for any neighboring pair $(v, w)$, we make the definitions 
\begin{equation*}
U_{v, w}^+ = [\eta_{a_i}, w] \cup (B(\{w\}, \nu)^+ \cap \Sigma_S), \ \ \ U_{v, w}^- = [\eta_{b_i}, w] \cup (B(\{w\}, \nu)^- \cap \Sigma_S),
\end{equation*}
\begin{equation*}
U_{v, w} = U_{v, w}^+ \cup U_{v, w}^- = \Lambda_{(i)} \cup [v, w] \cup (B(\{w\}, \nu) \cap \Sigma_S).
\end{equation*}
(Although the sets $U_{v, w}^\pm, U_{v, w}$ described above depend on the choice of $\nu$, we suppress it from the notation.)  It is an easy observation to note that given any $\nu > 0$, the convex hull $\Sigma_S$ coincides with the union of the subspaces $U_{v, w}$ over all neighboring pairs $(v, w)$: indeed, it contains the union of the subspaces $[a_i, w] \cup [b_i, w] = \Lambda_{(i)} \cup [v, w] \subset U_{v, w}$; this includes all axes $\Lambda_{(i)}$ (as there is at least one distinguished vertex of $\Sigma_{S, 0}$ lying in each $\Lambda_{(i)}$) as well as all paths between any pair of distinguished vertices lying in distinct axes and thus paths between any pair of axes and between any pair of points corresponding to elements of $S$.  Our strategy for proving the theorem is, after choosing an appropriate $\nu > 0$, to use \Cref{cor one segment} to describe the behavior of $\Theta_*$ restricted to each $U_{v, w}^-$ and $U_{v, w}^+$ and from that, describe it restricted to each $U_{v, w}$, and then ``glue these restrictions together'' to get a description of $\Theta_*$ on all of $\Sigma_S$.

Choose a neighboring pair $(v, w)$ with corresponding indices $i \neq j$ as in the construction of $U_{v, w}^\pm$.  Letting $\sigma_{v, w}^- \in \PGL_2(K)$ be the (unique) automorphism mapping $b_i, a_j, b_j \in \Omega_{\Gamma_0}$ respectively to $1, 0, \infty \in \sigma(\Omega)$, \Cref{lemma gluing} (putting $a' = a_i$ and $b' = b_i$) tells us that there is a fractional linear transformation $\tau_{v, w}^- \in \PGL_2(K)$ making the diagram (\ref{eq gluing commutative diagram}) commute.  On identifying the respective elements $\sigma_{v, w}^-(a_l), \sigma_{v, w}^-(b_l) \in \sigma(S)$ with the elements named $a_l, b_l$ in the statement of \Cref{cor one segment} for all indices $l$, we see that the hypotheses (i)-(iii) of \Cref{thm one segment} (and thus the hypotheses of \Cref{cor one segment}) are satisfied as they describe $(\eta_{\sigma_{v, w}^-(a_i)} \vee \eta_1, \mathfrak{v})$ as a neighboring pair of vertices of $\Sigma_{\sigma_{v, w}^-(S)}$.  It is clear that for any $\nu > 0$, the image $\sigma(U_{v, w}^-)$ coincides with the subspace $U^-$ defined in the statement of \Cref{cor one segment}; in fact, we have $\sigma([\eta_{b_i}, v]) = [\eta_1, \eta_1 \vee \eta_{1+\lambda}]$ and $\sigma([v, w]) = [\eta_1 \vee \eta_{1 + \lambda}, \mathfrak{v}]$.  Now, using the fact that the action of fractional linear transformations (in particular, $\tau_{v, w}^-$) on $\Berk$ is a metric-preserving homeomorphism, we use \Cref{cor one segment}(b) to describe the behavior of the function $\tilde{\Theta}_*$ restricted to $[\eta_{b_i}, w]$ as follows: 
\begin{itemize}
\item it maps $[\eta_{b_i}, v] = (\sigma_{v, w}^-)^{-1}([\eta_1, \eta_1 \vee \eta_{1+\lambda}])$ to the path $[(\tau_{v, w}^-)^{-1}(\eta_\infty), (\tau_{v, w}^-)^{-1}(\mathfrak{v})]$ in a manner that scales the metric by $p$; and 
\item it maps $[v, w] = (\sigma_{v, w}^-)^{-1}([\eta_1 \vee \eta_{1+\lambda}, \mathfrak{v}])$ to the path $[(\tau_{v, w}^-)^{-1}(\mathfrak{v}), (\tau_{v, w}^-)^{-1}(\eta_c \vee \eta_1)]$ (where $c = \Theta^{\Gamma_0^\sigma}_{1, 1+\lambda}(0)$) in a manner that scales the metric by $p$ on the sub-segments $[v, v']$ and $[w', w]$ and which preserves the metric on the sub-segment $[\tilde{v}, \tilde{w}]$, where $\tilde{v}$ (resp. $\tilde{w}$) is the unique point in the path $[v, w]$ of distance $\frac{v(p)}{p - 1}$ from $v$ (resp. $w$).
\end{itemize}
We also use \Cref{cor one segment}(c) to conclude that, if $\nu > 0$ is chosen small enough, we have 
\begin{equation} \label{eq disjointness -}
\tilde{\Theta}_*([\eta_{b_i}, w]) \cap (\tilde{\Theta}_*(U_{v, w}^- \smallsetminus [\eta_{b_i}, w]) = \varnothing.
\end{equation}

Now letting $\sigma_{v, w}^+ \in \PGL_2(K)$ be the (unique) automorphism mapping $a_i, b_j, a_j \in \Omega_{\Gamma_0}$ respectively to $1, 0, \infty \in \sigma(\Omega)$, \Cref{lemma gluing} (again putting $a' = a_i$ and $b' = b_i$) tells us that there is a fractional linear transformation $\tau_{v, w}^+ \in \PGL_2(K)$ making the diagram (\ref{eq gluing commutative diagram}) commute.  We now use \Cref{cor one segment} by identifying the respective elements $\sigma_{v, w}^+(a_i), \sigma_{v, w}^+(b_i) \in \sigma(S)$ with the elements named $b_i, a_i$ in the statement of the statement of that corollary for all indices $i$ (as before, except that we have now switched $a_i$ with $b_i$!), and by employing a completely analogous argument to the one above, to describe the behavior of the function $\tilde{\Theta}_*$ restricted to $[\eta_{a_i}, w]$ with respect to the metric to get a similar description as the one obtained above for its behavior on $[\eta_{b_i}, w]$.  This matches with the function's previously known behavior on $[v, w] = [\eta_{a_i}, w] \cap [\eta_{b_i}, w]$.  It also gives us the new information that $(\tilde{\Theta}_{a_i, b_i})_*$ maps $[\eta_{a_i}, v] = (\sigma_{v, w}^+)^{-1}([\eta_1, \eta_1 \vee \eta_{1+\lambda'}])$ (where $\lambda' = \sigma_{v, w}^+(b_i) - 1$) to the path $[\iota \circ (\tau_{v, w}^+)^{-1}(\eta_\infty), \iota \circ (\tau_{v, w}^+)^{-1}(\mathfrak{v})]$ (where $\iota$ is the reciprocal map) in a manner that scales the metric by $p$.  Finally, it is clear that for any $\nu > 0$, the image $\sigma(U_{v, w}^+)$ again coincides with the subspace $U^-$ defined in the statement of \Cref{cor one segment}, which means that if $\nu > 0$ is chosen small enough, we have 
\begin{equation} \label{eq disjointness +}
\tilde{\Theta}_*([\eta_{a_i}, w]) \cap \tilde{\Theta}_*(U_{v, w}^+ \smallsetminus [\eta_{a_i}, w]) = \varnothing.
\end{equation}
In particular, since the images $\tilde{\Theta}_*([\eta_{b_i}, w])$ and $\tilde{\Theta}_*([\eta_{a_i}, w])$ are each non-backtracking paths whose endpoints of Type I must respectively equal $\eta_{\tilde{\Theta}_*(a_i)}, \eta_{\tilde{\Theta}_*(b_i)} \in \proj_{\cc_K}^1$, we get that $\Theta^{\Gamma_0}_{a_i, b_i}$ maps the axis $\Lambda_{(i)}$ homeomorphically onto the axis $\Lambda_{\tilde{\Theta}_*(a_i), \tilde{\Theta}_*(b_i)}$.

It now follows directly that the behavior of the function $\tilde{\Theta}_*$ restricted to $\Lambda_{(i)} \cup [v, w]$ is as claimed in \Cref{thm main berk} with respect to the metric as well as being a homeomorphism onto its image $\Lambda_{\tilde{\Theta}(a_i), \tilde{\Theta}(b_i)} \cup [\tilde{\Theta}_*(v), \tilde{\Theta}_*(w)]$.  Combining (\ref{eq disjointness -}) and (\ref{eq disjointness +}), we moreover conclude that, for sufficiently small $\nu > 0$, we have 
\begin{equation} \label{eq disjointness +/-}
\tilde{\Theta}_*([v, w]) \cap \tilde{\Theta}_*(B(\{w\}, \nu) \smallsetminus [v, w]) = \varnothing.
\end{equation}
As the choice of neighboring pair $(v, w)$ was arbitrary, we have demonstrated these properties for all subspaces $U_{v, w}$.

Now we will show that $\tilde{\Theta}_*$ is one-to-one on a neighborhood of each of its distinguished vertices; as it has already been been shown to be one-to-one on each axis $\Lambda_{(i)}$ as well as each path $[v, w]$ between neighboring pairs of distinguished vertices, it will then follow from a straightforward exercise in topology concerning continuous maps between real trees that the function $\tilde{\Theta}_*$ is one-to-one on all of $\Sigma_S$.  We choose a distinguished vertex $w$ lying in an axis $\Lambda_{(j)}$ and a real number $\nu > 0$ small enough that $w$ is the only distinguished vertex in $B(\{w\}, \nu)$; it is clear from \Cref{dfn convex hull} that any neighborhood of $w$ in $\Sigma_S$ contains a star shape centered at $w$ with $2$ edges coming out of $w$ being ends of the paths $[\eta_{a_j}, w]$ and $[\eta_{b_j}, w]$ and with each other edge coming out of $w$ being the end of the path $[v, w]$ where $(v, w)$ is a neighboring pair.  We know from (\ref{eq disjointness +/-}) that, after possibly shrinking $\nu$, for any distinguished vertex $v$ such that $(v, w)$ is a neighboring pair, the image of $B(\{w\}, \nu) \cap [v, w]$ under $(\tilde{\Theta}^{\Gamma_0}_{a_i, b_i})_*$ is disjoint from the image of its complement in $B(\{w\}, \nu)$.  We also know (by applying results obtained above to the neighboring pair $(w, v)$ rather than $(v, w)$) that $\tilde{\Theta}_*$ is one-to-one when restricted to the axis $\Lambda_{(j)}$ and when restricted to the each of the paths $[\eta_{a_i}, w], [\eta_{b_i}, w] \subset \Sigma_S$.  All of this implies that the images of each edge coming out of $w$ in the star shape under $\tilde{\Theta}_*$ intersect only at the point $\tilde{\Theta}_*(w)$.  It follows that the function $\tilde{\Theta}_*$ is one-to-one on $B(w, \nu)$, as desired.

We finally have to show that $\tilde{\Theta}_*$ maps $\Sigma_S$ homeomorphically onto the convex hull $\Sigma_{\mathcal{B}}$.  From the fact that $\tilde{\Theta}_*$ is a homoeomorphism onto its image when restricted to each $U_{v, w}$ and that it is one-to-one on $\Sigma_S$, we see that it is a homeomorphism when restricted to $\Sigma_S = \bigcup U_{v, w}$ as well.  We have already noted that $\tilde{\Theta}_*$ maps axes to axes, or more precisely, that we have $\tilde{\Theta}_*(\Lambda_{(i)}) = \Lambda_{\tilde{\Theta}(a_i), \tilde{\Theta}(b_i)}$ for $0 \leq i \leq g$.  Meanwhile, we have seen that for any neighboring pair $(v, w)$ of distinguished vertices of $\Sigma_S$ with $v \in \Lambda_{(i)}$ and $w \in \Lambda_{(j)}$, the image $\tilde{\Theta}_*([v, w])$ is a (non-backtracking) path.  Then clearly the endpoints of the image lie in $\Lambda_{\tilde{\Theta}(a_i), \tilde{\Theta}(b_i)}$ and $\Lambda_{\tilde{\Theta}(a_j), \tilde{\Theta}(b_j)}$ while its interior does not intersect any axis $\Lambda_{\tilde{\Theta}(a_l), \tilde{\Theta}(b_l)}$.  Such a path must be contained in the convex hull $\Sigma_{\mathcal{B}}$ (as the convex hull is path connected and must contain the shortest path between each pair of axes), so we have $\tilde{\Theta}_*(\Sigma_S) \subseteq \Sigma_{\mathcal{B}}$.  Meanwhile, since the convex hull $\Sigma_S$ is connected, the image $\tilde{\Theta}_*(\Sigma_S)$ is also connected and we get the reverse inclusion $\tilde{\Theta}_*(\Sigma_S) \supseteq \Sigma_{\mathcal{B}}$.  This completes the proof of \Cref{thm main berk}.

\section{Cluster data of branch points of split degenerate superelliptic curves} \label{sec branch points}

An almost immediate corollary of \Cref{thm main berk} gives us a result concerning the branch locus of a split degenerate $p$-cover of the projective line.

\begin{cor} \label{cor branch points}

Let $C / K$ be a $p$-cyclic cover of $\proj_K^1$ which has split degenerate reduction, and denote its set of branch points by $\mathcal{B} \subset \proj_{\cc_K}^1$.  The set $\mathcal{B}$ is clustered in $\frac{pv(p)}{p - 1}$-separated pairs.

\end{cor}

\begin{proof}

Let $v, w$ be distinct distinguished vertices of $\Sigma_S$ which do not lie in the same axis $\Lambda_{(l)}$ for any index $l$; \Cref{prop clustered in pairs distinguished vertices} says that we have $\delta(v, w) > \frac{2v(p)}{p - 1}$.  Let $i \neq j$ be the indices such that $v \in \Lambda_{(i)}$ and $w \in \Lambda_{(j)}$, and let $\tilde{v}$ (resp. $\tilde{w}$) be the (unique) point in the path $[v, w]$ whose distance from $v$ (resp. $w$) is equal to $\frac{v(p)}{p - 1}$.  Then, using the terminology of \Cref{thm main berk}, we have $[v, \tilde{v}] \cup [\tilde{w}, w] \subset \llbracket v, w \rrbracket$.  As each of the segments $[v, \tilde{v}]$ and $[\tilde{w}, w]$ has length $\frac{v(p)}{p - 1}$, we get $\mu(v, w) \geq \frac{2v(p)}{p - 1}$.  Now \Cref{thm main berk} says that the images $\pi_*(v), \pi_*(w)$ of the distinguished vertices $v, w$ are themselves distinguished vertices of $\Sigma_{\mathcal{B}}$ and that we have 
\begin{equation}
\delta(\pi_*(v), \pi_*(w)) = \delta(v, w) + (p - 1)\mu(v, w) \geq \delta(v, w) + 2v(p) > \frac{2v(p)}{p - 1} + 2v(p) = \frac{2pv(p)}{p - 1}.
\end{equation}
As each distinguished vertex of $\Sigma_{\mathcal{B}}$ is the image under $\pi_*$ of a distinguished vertex of $\Sigma_S$, we have shown that the distance between any pair of distinguished vertices of $\Sigma_{\mathcal{B}}$ is $> \frac{2v(p)}{p - 1}$.  \Cref{prop clustered in pairs distinguished vertices} now tells us that the set $\mathcal{B}$ is clustered in $\frac{pv(p)}{p - 1}$-separated pairs.
\end{proof}

\begin{rmk} \label{rmk semistable model}

It should be possible to provide a purely geometric proof of \Cref{cor branch points}, based on the fact that a curve with split degenerate reduction over $K$, by definition, has a model over the ring of integers $\mathcal{O}_K$ of $K$ the components of whose special fiber are each a copy of the projective line $\proj_k^1$ over the residue field.

Moreover, the converse of \Cref{cor branch points} -- that the branch points of a $p$-cyclic cover of the projective line being clustered in $\frac{pv(p)}{p - 1}$-separated pairs implies split degenerate reduction -- is true, at least in the tame case.  This again can in principle be demonstrated by purely geometric arguments, and the idea of the proof is as follows.  Assume for simplicity that we have $\infty \in \mathcal{B}$ and that the (maximal) cluster $\mathfrak{s}_0 := \mathcal{B} \smallsetminus \{\infty\}$ of $\mathcal{B}$ has depth $d(\mathfrak{s}_0) = 0$ (these conditions can always be imposed after applying a suitable fractional linear transformation to $x$ and possibly replacing $K$ by a degree-$p$ extension).  Given such a curve $C$, we want to construct a model $\mathcal{C}$ of $C$ over the ring of integers $\mathcal{O}_K$.  In order to do so, we will construct a model $\mathcal{X} / \mathcal{O}_K$ of $\proj_K^1$ and let $\mathcal{C}$ be the normalization of $\mathcal{X}$ in the function field $K(C)$; the $p$-cyclic covering map $C \to \proj_K^1$ extends to a $p$-cyclic covering map $\mathcal{C} \to \mathcal{X}$.  It is well known that any model of $\proj_K^1$ is defined by a set of equations $\{x = c_i x_i + \alpha_i\}_{0 \leq i \leq t}$ for some elements $\alpha_i \in K$ and $c_i \in K^\times$; each coordinate $x_i$ corresponds to a component $\bar{X}_i$ of the special fiber.

Suppose first that $p$ is not the residue characteristic of $K$.  Let $\mathcal{B} \smallsetminus \{\infty\} =: \mathfrak{s}_0, \mathfrak{s}_1, \dots, \mathfrak{s}_t$ be the non-singleton clusters of $\mathcal{B}$, and for $1 \leq i \leq t$, choose an element $\alpha_i \in \mathfrak{s}_i$ and a scalar $c_i \in K^\times$ satisfying $v(c_i) = d(\mathfrak{s}_i)$, setting $c_0 = 1$ so that $x_0$ is just the standard coordinate $x$.  Then the desired model $\mathcal{X} / \mathcal{O}_K$ is defined by the set of equations $\{x_0 = c_i x_i + \alpha_i\}_{1 \leq i \leq t}$: one can show that the normalization $\mathcal{C}$ of $\mathcal{X}$ in the function field $K(C)$ is semistable.  (In fact, with a little work, one can see that the special fiber of $\mathcal{X}$ is isomorphic over the residue field $k$ to the image of the function $R_S$ as defined in \S\ref{sec GvanderP position}.)  The ($K$-)points $\alpha \in \mathcal{B} \subset \proj_K^1$ extend to $\mathcal{O}_K$-points $\underline{\alpha}$ of $\mathcal{X}$ which, by slight abuse of terminology, we will also refer to as elements of $\mathcal{B}$.  For each index $i$, the points $\alpha \in \mathcal{B}$ intersecting the component $\bar{X}_i$ are exactly the elements of $\mathfrak{s}_i$ which do not lie in any proper non-singleton sub-cluster of $\mathfrak{s}_i$.  It then follows from the property of being clustered in pairs that each component $\bar{X}_i$ of the special fiber of $\mathcal{X}$ intersects with exactly $0$, $1$, or $2$ elements of $\mathcal{B}$; these cases happen respectively when $\mathfrak{s}_i$ is the union of $\geq 2$ even-cardinality clusters, when $\mathfrak{s}_i$ has odd cardinality, and when $\mathfrak{s}_i$ does not satisfy either of the above two properties.  In the case that a component $\bar{X}_i$ does not intersect with any branch points, the cluster $\mathfrak{s}_i$ is the union of $\geq 2$ even-cardinality clusters, and there are exactly $p$ components of the special fiber of $\mathcal{C}$ (each isomorphic to $\proj_k^1$) mapping to the component $\bar{X}_i$ (with no ramification) under the $p$-cyclic covering map $\mathcal{C} \to \mathcal{X}$.  In the other two cases, there is exactly $1$ component $C_i$ of the special fiber of $\mathcal{C}$ mapping to $\bar{X}_i$ ramified at $2$ points, which by Riemann-Hurwitz implies that $C_i$ is again isomorphic to $\proj_k^1$.  In this way we see that the components of the special fiber of $\mathcal{C}$ are each isomorphic to $\proj_k^1$ and so $C$ has split degenerate reduction over $K$.

It is expected that via a more complicated construction of a semistable model $\mathcal{C}$ of $C$ exhibiting split degenerate reduction, one can prove the converse also in the wild case.  For instance, it can be done for hyperelliptic curves (\textit{i.e.} when $p = 2$) using methods found in the author's preprint \cite{fiore2023clusters} (collaborated with Leonardo Fiore).  In fact, under the hypothesis that $\mathcal{B}$ is clustered in $2v(2)$-separated pairs and using methods in \cite[\S6]{fiore2023clusters} and in particular results from \S6.4 of that paper, one may compute that the \emph{valid discs} (see \cite[Definition 5.11]{fiore2023clusters}) are precisely the discs corresponding to the points $v$ of the convex hull $\Sigma_{\mathcal{B}} \subset \Berk$ which satisfy $\delta(w, \Lambda_{(i)}) = 2v(2)$ for some index $i$ as well as the discs corresponding to non-distinguished vertices not lying in the tubular neighborhood $B(\Lambda_{(i)}, 2v(2))$ for any index $i$.  (We suspect that when $p \geq 3$, an identical statement holds, with $2v(2)$ replaced by $\frac{pv(p)}{p - 1}$.)  In this context, a \emph{valid disc} is by definition given by $\{z \in \cc_K \ | \ v(z - \alpha) \geq v(c)\}$ for some $\alpha \in \cc_K$ and $c \in \cc_K^\times$ such that the corresponding coordinate $x'$ (with $x = cx + \alpha$) defines one of the components of the special fiber of the desired model $\mathcal{X}$ of $\proj_K^1$.  If $\bar{X}$ is the component of the special fiber of $\mathcal{X}$ corresponding as above to a (non-distinguished) vertex of $\Sigma_{\mathcal{B}}$ (in which case the cluster $\mathfrak{s}$ corresponding to it via \Cref{prop vertex cluster correspondence}(b) is the union of $\geq 2$ even-cardinality sub-clusters, or is \emph{\"{u}bereven} as in \cite[Definition 8.6]{fiore2023clusters}), then there are exactly $2$ components of the special fiber of $\mathcal{C}$ (each isomorphic to $\proj_k^1$) mapping to the component $\bar{X}$ (with no ramification) under the $2$-covering $\mathcal{C} \to \mathcal{X}$.  Meanwhile, for all other components of the special fiber of $\mathcal{X}$, there is exactly $1$ component of the special fiber of $\mathcal{C}$ mapping to it, which is ramified at $1$ point, and \cite[Proposition 4.28, Proposition 6.17(c)(d)]{fiore2023clusters} can be used to show that this component is also isomorphic to $\proj_k^1$.  It follows again that $C$ has split degenerate reduction over $K$.

\end{rmk}

\begin{rmk} \label{rmk elliptic}

In the case that $g = 1$, every subset $S \subset \proj_K^1$ which is clustered in pairs is not only $p$-superelliptic but optimal by \cite[Proposition 3.25]{yelton2024branch}, and one can easily and quickly use \Cref{thm main berk} to describe the cluster data of the set $\mathcal{B}$ of branch points of the resulting superelliptic curve: the convex hull $\Sigma_S$ contains the axes $\Lambda_{(0)}, \Lambda_{(1)}$ at some distance $\delta > \frac{2v(p)}{p - 1}$ apart (equivalently, there is a cluster of $S$ of cardinality $2$ and relative depth $\delta > \frac{2v(p)}{p - 1}$), while the convex hull $\Sigma_{\mathcal{B}}$ contains the images of these axes at a distance of $\delta + 2v(p) > \frac{2pv(p)}{p - 1}$ apart (equivalently, if $\infty \in S$, the image of the cluster of $S$ of cardinality $2$ is a cluster of $\mathcal{B}$ with relative depth $\delta + 2v(p) > \frac{2pv(p)}{p - 1}$).  \Cref{cor branch points} says that the set of branch points of any split degenerate cyclic $p$-cover of $\proj_K^1$ of genus $(p - 1)g$ satisfies this property.  This is already known in the case of an elliptic curve (\textit{i.e.} when $g = 1$ and $p = 2$), in which case the split degenerate reduction condition is known as split multiplicative reduction: see \cite[Remark 9.5(a)]{fiore2023clusters} or the results of \cite{yelton2021semistable}, for instance.

\end{rmk}






\bibliographystyle{plain}
\bibliography{bibfile}

\end{document}